\documentclass[10pt]{amsart}
\usepackage[T1]{fontenc}
\usepackage[utf8]{inputenc}
\usepackage{amsmath}
\usepackage{amsxtra}
\usepackage{amscd}
\usepackage{amsthm}
\usepackage{faktor}
\usepackage{amsfonts}
\usepackage{amssymb}
\usepackage{eucal}
\usepackage{mathabx}
\usepackage[matrix,arrow,curve]{xy}
\usepackage{pb-diagram}
\usepackage{comment}
\usepackage{array}
\usepackage{accents}
\newcommand{\fg}{{\mathfrak{g}}}
\newcommand{\bp}{{\mathbb{P}}}
\newcommand{\fh}{{\mathfrak{h}}}

\newcommand{\fn}{{\mathfrak{n}}}

\newcommand{\fz}{{\mathfrak{z}}}
\newcommand{\fsl}{{\mathfrak{sl}}}
\newcommand{\fgl}{{\mathfrak{gl}}}
\newcommand{\tr}{{\rm tr} \,}

\newcommand{\Tr}{{\rm Tr} \,}

\newcommand{\End}{{\rm End} \,}
\newcommand{\Id}{{\rm Id} \,}
\newcommand{\Ad}{{\rm Ad} \,}
\newcommand{\rk}{{\rm rk} \,}
\newcommand{\gr}{{\rm gr} \,}

\newcommand{\bc}{\mathbb{C}}

\newcommand{\diag}{{\rm diag} \,}
\DeclareMathOperator{\spann}{span}
\DeclareMathOperator{\ad}{ad}
\DeclareMathOperator{\Spec}{Spec}
\DeclareMathOperator{\Res}{Res}
\DeclareMathOperator{\bigrr}{bigr}
\newcommand{\wt}{\widetilde}
\newcommand{\dbtilde}[1]{\accentset{\approx}{#1}}

\newtheorem{defn}[subsection]{Definition}
\newtheorem{thm}[subsection]{Theorem}
\newtheorem{lem}[subsection]{Lemma}
\newtheorem{prop}[subsection]{Proposition}

\newtheorem{cor}[subsection]{Corollary}

\newtheorem*{rem}{Remark}

\title{On classical limits of Bethe subalgebras in Yangians}
\author{Aleksei Ilin and Leonid Rybnikov}

\begin{document}
\maketitle

\begin{center} \emph{To the memory of Ernest Borisovich Vinberg}
\end{center}

\begin{abstract} The Yangian $Y(\fg)$ of a simple Lie algebra $\fg$ can be regarded as a deformation of two different Hopf algebras: the universal enveloping algebra of the current algebra $U(\fg[t])$ and the coordinate ring of the first congruence subgroup $\mathcal{O}(G_1[[t^{-1}]])$. Both of these algebras are obtained from the Yangian by taking the associated graded with respect to an appropriate filtration on $Y(\fg)$.

Bethe subalgebras $B(C)$ in $Y(\fg)$ form a natural family of commutative subalgebras depending on a group element $C$ of the adjoint group $G$. The images of these algebras in tensor products of fundamental representations give all integrals of the quantum XXX Heisenberg magnet chain. 

We describe the associated graded of Bethe subalgebras in the Yangian $Y(\fg)$ of a simple Lie algebra $\fg$ as subalgebras in $U(\fg[t])$ and in $\mathcal{O}(G_1[[t^{-1}]])$ for all semisimple $C\in G$. In particular, we show that associated graded in $U(\fg[t])$ of the Bethe subalgebra $B(E)$ assigned to the unity element of $G$ is the universal Gaudin subalgebra of $U(\fg[t])$ obtained from the center of the corresponding affine Kac-Moody algebra $\hat{\fg}$ at the critical level. This generalizes Talalaev's formula for generators of the universal Gaudin subalgebra to $\fg$ of any type. In particular, this shows that higher Hamiltonians of the Gaudin magnet chain can be quantized without referring to the Feigin-Frenkel center at the critical level.

Using our general result on associated graded of Bethe subalgebras, we compute some limits of Bethe subalgebras corresponding to regular semisimple $C\in G$ as $C$ goes to an irregular semisimple group element $C_0$. We show that this limit is the product of the smaller Bethe subalgebra $B(C_0)$ and a quantum shift of argument subalgebra in the universal enveloping algebra of the centralizer of $C_0$ in $\fg$. This generalizes the Nazarov-Olshansky solution of Vinberg's problem on quantization of (Mishchenko-Fomenko) shift of argument subalgebras. 
\end{abstract}

\section{Introduction}

\subsection{Yangians and Bethe subalgebras.}
Let $\fg$ be a simple complex Lie algebra, and $G$ be the corresponding adjoint group. The Yangian $Y(\fg)$ is the unique homogeneous Hopf algebra deformation of the universal enveloping algebra $U(\fg[t])$, see \cite{drin}. It is also a Hopf algebra deformation of the algebra $\mathcal{O}(G_1[[t^{-1}]])$ of functions on first congruence subgroup $G_1[[t^{-1}]]\subset G[[t^{-1}]]$ deforming the natural Poisson structure on $G_1[[t^{-1}]]$, see \cite{kamn}.

%$Y(\fg)$ has at least three different realizations by generators and relations: $J$-realization, current (new) realization and $RTT$-realization, see \cite{drin} and \cite{drin2}.

Bethe subalgebras are the family of commutative subalgebras of $B(C)\subset Y(\fg)$ depending on a group element $C\in G$. The particular cases of Bethe subalgebras were defined in \cite{cerednik}, \cite{drin3}, \cite{kr}, \cite{mo}, \cite{molev1} and \cite{nazol}. The most general definition of Bethe subalgebras goes back to Drinfeld: namely, one can define $B(C)$ as the subalgebra generated by all Fourier coefficients of $Tr_V (\rho(C)\otimes 1) (\rho\otimes\Id) (\mathcal{R}(u))$ for all finite dimensional representations $\rho: Y(\fg)\to \End(V)$, where $\mathcal{R}(u)$ is the universal $R$-matrix with spectral parameter. In  \cite{ir2} we gave a detailed description of these subalgebras using the $RTT$-realization of the Yangian from \cite{drin} and \cite{wend}. 
%Our definition involves the $RTT$-realization of $Y(\fg)$ (or, equivalently, it involves the universal $R$-matrix for $Y(\fg)$).
%This construction involves the universal $R$-matrix, so it does not give any explicit formulas immediately, but it is useful to study limiting properties of Bethe subalgebras.

\subsection{The (universal) Gaudin subalgebra.} The universal enveloping algebra of the current
algebra $\mathfrak{g}[t]$
contains a large commutative subalgebra $\mathcal{A}_{\fg}\subset
U(\mathfrak{g}[t])$. This subalgebra comes from the center of the
universal enveloping of the affine Kac--Moody algebra
$\hat{\mathfrak{g}}$ at the critical level and gives rise to the
construction of higher hamiltonians of the Gaudin model (due to
Feigin, Frenkel and Reshetikhin, \cite{ffr}). Though there are no explicit
formulas for the generators of $\mathcal{A}_{\fg}$ known in general, the
classical analogue of this subalgebra, i.e. the associated
graded subalgebra in the Poisson algebra $A_{\fg}\subset
S(\mathfrak{g}[t])$, can be easily described. Namely, it is generated by all Fourier components of the $\mathbb{C}[[t^{-1}]]$-valued functions $\Phi_l(x(t))$ on $t^{-1}\mathfrak{g}[[t^{-1}]]= \Spec S(\mathfrak{g}[t])$ for $\Phi_l$ being free homogeneous generators of the algebra of adjoint invariants $S(\mathfrak{g})^\mathfrak{g}\subset S(\fg)$ (here we identify $t^{-1}\mathfrak{g}[[t^{-1}]]$ with $\mathfrak{g}[t]^*$ via the $\mathfrak{g}$-invariant pairing $t^{-1}\mathfrak{g}[[t^{-1}]]\times \mathfrak{g}[t]\to\mathbb{C}$ given by $(x(t),y(t)):=\Res_{t=0}x(t)y(t)dt$. 

\subsection{The associated graded of a Bethe algebra.} Let $C$ be any element of a maximal torus $T\subset G$. Denote by $\fz_\fg(C)$ the centralizer of $C$ in the Lie algebra $\fg$. It is a reductive Lie algebra containing the Cartan subalgebra $\fh \subset \fg$. The generators of Bethe subalgebra $B(C)\subset Y(\fg)$ are invariant with respect to the adjoint action of $\fz_\fg(C)$.

\medskip

\noindent {\bf Theorem A.} \emph{\begin{itemize} \item The associated graded of $B(C)$ in $U(\fg[t])$ is the universal Gaudin subalgebra $\mathcal{A}_{\mathfrak{z}_\fg(C)}\subset U(\mathfrak{z}_\fg(C)[t])\subset U(\fg[t])$;
\item The associated graded of $B(C)$ in $\mathcal{O}(G_1[[t^{-1}]])$ is generated by all Fourier coefficients of the $\mathbb{C}[t^{-1}]$-valued functions $\sigma(C)(g(t)):= Tr_V C\cdot g(t)$ (where $g(t)\in G_1[[t^{-1}]]$) for all finite dimensional $G$-modules $V$; \item The Bethe subalgebra $B(C)$ is a maximal commutative subalgebra in $Y(\fg)^{\mathfrak{z}_\fg(C)}$.
\end{itemize}} 

\medskip

In particular, this gives a construction of the universal Gaudin subalgebra independent of the representation theory of $\hat{\fg}$ at the critical level and for arbitrary simple $\fg$. We expect this leads to explicit type-free formulas for higher Gaudin Hamiltonians generalizing those of Talalaev, Chervov and Molev, see \cite{talalaev}, \cite{cm}.

\begin{rem}\emph{ We believe that our Theorem A is a part of a more general picture describing all possible degenerations of the affine quantum group $U_q(\hat{\fg})$ at the critical level. In particular, according to Ding and Etingof \cite{de} the center of $U_q(\hat{\fg})$ at the critical level is generated by traces of the $R$-matrix, so it is natural to expect that both Bethe subalgebras in the Yangian and Gaudin subalgebras are degenerate versions of this center. We hope to return to this in forthcoming papers. }
\end{rem}

\subsection{Limit Bethe subalgebras.}
Let $T^{reg} \subset T$ be the set of regular elements of the torus $T$.
From Theorem A we see that the family of Bethe subalgebras $B(C)\subset Y(\fg)$ is not flat, i.e. the Poincar\'e series of $B(C)$ is not constant in $C\in T$, because for non-regular $C\in T\backslash T^{reg}$, the subalgebra $B(C)$ becomes smaller. On the other hand a natural way to assign a commutative subalgebra of the same size as for $C \in T^{reg}$ to any $C_0\in T\setminus T^{reg}$ by taking some {\em limit} of $B(C)$ as $C\to C_0$ (this idea goes back to Vinberg \cite{vinberg} and Shuvalov \cite{shuvalov}). In general, such limit subalgebra $\lim\limits_{C\to C_0}B(C)$ is not unique since it depends on the path $C(\varepsilon)$ such that $C(0)=C_0$. The second goal of this paper is to study the simplest limits of Bethe subalgebras corresponding to $C(\varepsilon) = C_0 \exp(\varepsilon \chi), C_0 \in T \setminus T^{reg}, \chi \in \fh$ as $\varepsilon\to0$. It turns out that the resulting commutative subalgebra is the product of $B(C_0)$ and the quantum \emph{shift of argument algebra} in the universal enveloping algebra $U(\fz_\fg(C_0))\subset Y(\fg)$.

\subsection{Shift of argument subalgebras and Vinberg's problem.} The shift of argument subalgebras defined by Mishchenko and Fomenko in \cite{mf} are (generically) maximal Poisson commutative subalgebras in $S(\fg)$. For any $\chi\in\fg^*$ the corresponding shift of argument subalgebra $A_\chi\subset S(\fg)$ can be described as the subalgebra generated by all the derivatives along $\chi$ of all adjoint invariant in $S(\fg)$. More precisely, it is generated by all elements of the form $\partial_\chi^k\Phi_l$, for all generators $\Phi_l\in S(\fg)^\fg$, $l=1,\ldots \rk\fg,\ k=0,1,\ldots,m_l$, where $m_l=\deg\Phi_l-1$ are the exponents of the Lie algebra $\fg$. Then the number of generators is $\sum\limits_{l=1}^{\rk\fg}(m_l+1)=\frac{1}{2}(\dim \fg+\rk\fg)$, which is the maximal possible transcendence degree for Poisson commutative subalgebras in $S(\fg)$. 

\emph{Vinberg's problem} stated in \cite{vinberg} is the problem of lifting the Poisson commutative subalgebras $A_\chi\subset S(\fg)$ to commutative subalgebras in the universal enveloping algebra $U(\fg)$. In \cite{nazol} Olshansky and Nazarov construct the lifting of a shift of argument subalgebra $A_\chi$ to $U(\fg)$ as the image of Bethe subalgebra $B(\chi)$ in the (twisted) Yangian of $\fg$ under the evaluation homomorphism to $U(\fgl_n)$. This works only for classical $\fg$ since for others there is no evaluation homomorphism from the Yangian to $U(\fg)$. 

In \cite{ryb06} Vinberg's problem was solved affirmatively for arbitrary simple $\fg$ and semisimple $\chi$ with the help of the Feigin-Frenkel center of $U(\hat{\fg})$ at the critical level. Namely, the lifting $\mathcal{A}_\chi\subset U(\fg)$, called \emph{quantum} shift of argument subalgebra, was determined as the image of (a version of) the universal Gaudin subalgebra under some homomorphism. Moreover, it was proved that, for generic $\chi$, the subalgebras $A_\chi\subset S(\fg)$ can be lifted to the universal enveloping algebra $U(\fg)$ \emph{uniquely}.  

Our second main result is the following

\medskip

\noindent {\bf Theorem B.}
%про предельные подалгебры
\emph{Let $C(\varepsilon) = C_0 \exp(\varepsilon \chi), C_0 \in T \setminus T^{reg}$ with $\chi \in \fh\subset \fz_\fg(C_0)$ being a generic regular semisimple element of the centralizer of $C_0$ (i.e. belonging to some complement of countably many proper closed subsets in $\fh^{reg}$).
Then $$\lim_{\varepsilon \to 0} B(C(\varepsilon)) = B(C_0) \otimes_{Z(U(\fz_{\fg}(C_0)))} \mathcal{A_{\chi}},$$}
where $\mathcal{A_{\chi}} \subset U(\fz_{\fg}(C_0))$ is the quantum shift of argument subalgebra corresponding to $\chi$.

\medskip

\begin{rem} \emph{Theorem B can be regarded as the closest approximation to the Olshansky-Nazarov solution of Vinberg's problem for arbitrary simple $\fg$: indeed, now one can \emph{define} the lifting of $A_\chi\subset S(\fg)$ to the universal enveloping algebra as $\mathcal{A}_\chi:=U(\fg)\cap \lim_{\varepsilon \to 0} B(C(\varepsilon))$ for $C(\varepsilon)=\exp(\varepsilon \chi)$.}
\end{rem}

\subsection{} The paper is organized as follows. In section 2 we study two classical limits of the Yangian and relations between them. In section 3 we define Bethe subalgebras and give the lower bound for the size of a Bethe subalgebra. In section 4 we define the universal Gaudin subalgebra and study some its properties. In section 5 we prove Theorem A. In section 6 we study some limits of Bethe subalgebras and prove Theorem B.

\subsection{Acknowledgements} We thank Boris Feigin for stimulating discussions. The study has been funded within the framework of the HSE University Basic Research Program. Both authors were supported in part by the RFBR grant 20-01-00515. Theorem~B was proved under support of the RSF grant 19-11-00056. The first author is a Young Russian Mathematics award winner and would like to thank its sponsors and jury. The work of the second author was supported by the Foundation for the Advancement of Theoretical Physics and Mathematics “BASIS”.

\section{Two classical limits of the Yangian}

\subsection{Notations and definitions.}
Let $\fg$ be a complex simple Lie algebra, $G$ be the corresponding connected adjoint group, $\tilde G$ be the corresponding connected simply-connected group. Let $T \subset G$ be a maximal torus, $T^{reg} \subset T$ be the set of regular elements of $T$. Let $\fh \subset \fg$ be tangent Cartan subalgebra of $T$. Let $\left<\cdot,\cdot\right>$ be the Killing form on $\fg$ and $\{x_a\}$, $a=1,\ldots,\dim \fg$, be an orthonormal basis of $\fg$ with respect to $\left<\cdot,\cdot\right>$. We identify $\fg^*$ with $\fg$ using the Killing form.  
%Let $\Phi$ be the corresponding to the Lie algebra $\fg$ root system and $\Delta$ be a fixed set of simple roots. 
Let $m_i, i=1, \ldots, \rk \fg$ be the set of exponents of Lie algebra $\fg$. Let $\mathcal{O}(G)$ and $\mathcal{O}(\tilde G)$ be the algebras of polynomial functions on $G$ and $\tilde G$ respectively. 

Let $Y(\fg)$ be the Yangian of $\fg$. 
Let $V = \bigoplus_{i=1}^{\rk \fg} V(\omega_i,0)$ be the direct sum of fundamental representations of $Y(\fg)$. Let $R(u-v)$ be the image of the universal $R$-matrix in $\End(V)^{\otimes 2}$. Using this data we define the $RTT$-realization $Y_V(\fg)$ as follows. It turns out that $Y_V(\fg) \simeq Y(\fg)$, see \cite{drin} and \cite{wend} for details.

\begin{defn}
The Yangian $Y_V(\fg)$ is a unital associative algebra generated by the elements
$t_{ij}^{(r)},  1 \leq i, j \leq \dim V; r \geq 1$  with the defining relations 
$$R(u-v) T_1(u) T_2(v) = T_2(v) T_1(u) R(u-v) \quad \text{in} \ \End(V)^{\otimes 2} \otimes Y_V(\fg)[[u^{-1}, v^{-1}]],$$
$$S^2(T(u))= T(u+\frac{1}{2}c_{\fg}),$$
where $S(T(u))=T(u)^{-1}$ is the antipode map and
$c_{\fg}$ is the value of the Casimir element of $\fg$ on the adjoint representation.

Here $$T(u) = [t_{ij}(u)]_{i,j = 1, \ldots, \dim V} \in \End V \otimes Y_V(\fg),$$
$$t_{ij}(u) = \delta_{ij} + \sum_r t_{ij}^{(r)} u^{-r}$$ and $T_1(u)$ (resp. $T_2(u)$) is the image of $T(u)$ in the first (resp. second) copy of $\End V$.
\end{defn}

%Yangian definition
\subsection{Two filtrations on the Yangian.}

The first filtration $F_1$ on $Y_V(\fg)$ is determined by putting $\deg t_{ij}^{(r)} = r$. More precisely, the $r$-th filtered component $F_1^{(r)}Y(\fg)$ is the linear span of all monomials $t_{i_1j_1}^{(r_1)}\cdot\ldots\cdot t_{i_mj_m}^{(r_m)}$ with $r_1+\ldots+r_m \leq r$.
By $\gr_1$ we denote the operation of taking associated graded algebra with respect to $F_1$. From the defining relations we see that $\gr_1 Y(\fg)$ is a commutative algebra. Moreover, we have a Poisson algebra isomorphism $\gr_1 Y_V(\fg) \simeq \mathcal{O}(G_1[[t^{-1}]])$ where the grading on $\mathcal{O}(G_1[[t^{-1}]])$ is given by the $\bc^*$ action dilating $t$ (see Section~\ref{ss:congruence} for details):

\begin{thm}\label{th:F1} \cite[Proposition 2.24]{ir2} There is an isomorphism of graded Poisson algebras $\gr_1 Y_V(\fg) \simeq \mathcal{O}(G_1[[t^{-1}]])$.
\end{thm}

\begin{cor}
\label{th:F1}
Poincar\'e series of $\gr_1 Y_V(\fg)$ is $\prod\limits_{r=1}^\infty (1-q^r)^{-\dim\fg}$.
\end{cor}

The second filtration $F_2$ on $Y_V(\fg)$ is determined by putting $\deg t_{ij}^{(r)} = r-1$. Similarly, the $r$-th filtered component $F_2^{(r)}Y(\fg)$ is the linear span of all monomials $t_{i_1j_1}^{(r_1)}\cdot\ldots\cdot t_{i_mj_m}^{(r_m)}$ with $r_1+\ldots+r_m \leq r+m$. 
By $\gr_2$ we denote the operation of taking associated graded algebra with respect to $F_2$. 

\begin{thm}\label{th:F2}\cite{wend}
$\gr_2 Y_V(\fg) \simeq U(\fg[t])$ where the grading is given by the $\bc^*$ action dilating $t$. Moreover, we have $t^{r-1}\fg\subset \spann \{ t_{ij}^{(r)}\} / F_2^{(r-2)}Y(\fg)$.
\end{thm}

\subsection{The associated bigraded algebra.}

The filtration $F_1$ on $Y(\fg)$ produces a filtration on $U(\fg[t])=\gr_2 Y(\fg)$ which we denote by the same letter $F_1$. Similarly, the filtration $F_2$ on $Y(\fg)$ descends to a filtration $F_2$ on $\mathcal{O}(G_1[[t^{-1}]])=\gr_1 Y(\fg)$. The corresponding associated graded algebras $ \gr_1\gr_2 Y(\fg)$ and $\gr_2\gr_1 Y(\fg)$ get a bigrading from the filtrations $F_1$ and $F_2$. We begin with some general facts about algebras with multiple filtrations.

For any algebra $A$ endowed with two filtrations, $F_1$ and $F_2$, one can define \emph{the associated bigraded algebra.} of $A$ as
$$
\bigrr A = \bigoplus_{i,j} \faktor{(F_1^{(i)}A \cap F_2^{(j)}A)}{(F_1^{(i-1)}A \cap F_2^{(j)}A + F_1^{(i)}A \cap F_2^{(j-1)}A)} 
$$

We also use the following notation: 
$$\gr_{12} A = \gr_1 \gr_2 A, \quad \gr_{21} A = \gr_2 \gr_1 A$$

\begin{lem}\label{le:bigr}
The associated bigraded algebra $\bigrr A$ is canonically isomorphic to $\gr_{12} A$ and to $\gr_{21}  A$.
\end{lem}
\begin{proof}
Consider the algebra $\gr_1 A = F_1^{(0)}A \oplus \faktor{F_1^{(1)}A}{F_1^{(0)}A} \oplus \ldots$. The filtration $F_2$ produces a filtration $W_0 \subset W_1 \subset \ldots$ on $\gr_1 A$ such that $W_i = \oplus_j \faktor{F_2^{(i)}A \cap F_1^{(j)}A}{F_2^{(i)}A \cap F_1^{(j-1)}A} $.

Note that $$\faktor{F_2^{(i)}A \cap F_1^{(j)}A}{F_2^{(i)}A \cap F_1^{(j-1)}A} \simeq \faktor{F_2^{(i) }A \cap F_1^{(j)} A+ F_1^{(j-1)} A} {F_1^{(j-1)}A}$$ therefore $W_0 \subset W_1 \subset \ldots$ is indeed a filtration.

We have the following canonical isomorphisms
\begin{eqnarray*} \faktor{W_i}{W_{i-1}}= \oplus_j \frac{\faktor{F_2^{(i)}A \cap F_1^{(j)}A}{F_2^{(i)}A \cap F_1^{(j-1)}A}}{\faktor{F_2^{(i-1)}A \cap F_1^{(j)}A}{F_2^{(i-1)}A \cap F_1^{(j-1)}A}}=\\
=\oplus_j\frac{F_2^{(i)}A \cap F_1^{(j)}A}{F_2^{(i-1)}A \cap F_1^{(j)}A+F_2^{(i)}A \cap F_1^{(j-1)}A}.
\end{eqnarray*}

%\begin{eqnarray*}\faktor{(F_2^{(i)}A \cap F_1^{(j)}A + F_1^{(j-1)}A)}{F_1^{(j-1)}A} \simeq \faktor{(F_2^{(i)}A \cap F_1^{(j)}A)}{(F_2^{(i)}A \cap F_1^{(j-1)}A)} = \faktor{(F_2^{(i)}A \cap F_1^{(j)}A)}{(F_2^{(i)}A \cap F_1^{(j)}A) \cap (F_2^{(i+1)}A \cap F_1^{(j-1)}A)} \simeq \\ \faktor{(F_2^{(i)}A \cap F_1^{(j)}A + F_2^{(i+1)}A \cap F_1^{(j-1)}A)}{(F_2^{(i+1)}A \cap F_1^{(j-1)}A)} \end{eqnarray*}

Then associated graded algebra $\gr_{21} A = W_0 \oplus \faktor{W_1}{W_0} \oplus \ldots$ is canonically isomorphic to $\bigrr A$. It is also isomorphic to $\gr_{12} A$ by the same argument.
\end{proof}

In contrast with last Lemma, if $U \subset A$ is a subspace, it is not true in general that $\gr_{12} U = \gr_{21} U$ as subspaces of $\bigrr A$, since the associated homomorphism of bigraded spaces $\bigrr U\to\bigrr A$ is not necessarily injective. Indeed, consider the algebra $A=\bc[x,y]$ with two filtrations setting by $\deg_1 x = 1,\ \deg_1 y=0$ and $\deg_2 x=0,\ \deg_2 y = 1$ and take $U = \bc[x+y]$. Then $\bigrr U$ is the polynomial algebra with one generator $z$ of the degree $(1,1)$ (coming from $x+y$). We have $x+y\in F_1^{(1)}A\cap F_2^{(0)}A+F_1^{(0)}A\cap F_2^{(1)}A$, so the image of $z$ in $\frac{F_1^{(1)}A\cap F_2^{(1)}A}{F_1^{(1)}A\cap F_2^{(0)}A+F_1^{(0)}A\cap F_2^{(1)}A}$ is zero. Hence the image of $\bigrr U$ in $\bigrr A$ is $\bc\cdot1$. At the same time $\gr_{12} U = \bc[x],\ \gr_{21} U = \bc[y]$. 

On the other hand, the following is still true:
\begin{prop}
\label{subgr}
Let $U$ be a vector subspace of $A$ such that $\gr_{12} U \subset \gr_{21} U$ as subspaces of $\bigrr A$. Then $\gr_{12} U = \gr_{21} U$.
\end{prop}
\begin{proof}
Suppose that we have an element $x \in \gr_{21} W \setminus \gr_{12} W$. Suppose that $\deg x = (k,l)$. Let $\dbtilde x$ be a lifting of $x$ to $U\subset A$. Then $\dbtilde x \in F_1^{(k)}A \cap F_2^{(l)}A + \sum\limits_{k'<k} F_1^{(k^{\prime})}A \cap F_2^{(l^{\prime})}A$, where $l'$ are some integers, so there is a presentation $\dbtilde x = \sum\limits_{i=1}^N \dbtilde x_i$ such that $\dbtilde x_i\in F_1^{(k_i)}A\cap F_2^{(l_i)}A$ with $k_1=k,\ k_{i+1}<k_i$. Take such a presentation of $\dbtilde x$ with the string $(k_1,k_2,k_3,\ldots,k_N)$ being lexicographically minimal among all such presentations. Then we have $\dbtilde x_i\not\in F_1^{(k_i-1)}A\cap F_2^{(l_i)}A+ F_1^{(k_i)}A\cap F_2^{(l_i-1)}A$ and $l_{i+1}>l_i$ for all $i$. Moreover, we can assume that $\dbtilde x$ is a lifting of  with the lexicographically minimal possible $(k_1,k_2,k_3,\ldots,k_N)$ among all liftings of $x$ to $U$. It is sufficient to show that $N=1$: indeed, then $\dbtilde x \in F_1^{(k)}A \cap F_2^{(l)}A$ and $\gr_{21} \dbtilde x = x \in \gr_{12} A$.

Suppose that $N>1$ and consider $y = \gr_{12} \dbtilde x$. It has degree $(k_N, l_N)$. Let $\dbtilde y\in U$ be a lifting of $y$ as an element of  $\gr_{21} A$, i.e. $\dbtilde y\in F_1^{(k_N)}A \cap F_2^{(l_N)}A+\sum\limits_{k'<k_N} F_1^{(k^{\prime})}A \cap F_2^{(l^{\prime})}A$. Then, in the same way as before, we have $\dbtilde y= \sum\limits_{i=N}^M \dbtilde y_i$ such that $\dbtilde y_i\in F_1^{(k_i)}A\cap F_2^{(l_i)}A$ with $k_{i+1}<k_i$ and $\dbtilde y_i\not\in F_1^{(k_i-1)}A\cap F_2^{(l_i)}A+ F_1^{(k_i)}A\cap F_2^{(l_i-1)}A$. 

So $\dbtilde x - \dbtilde y=\sum\limits_{i=1}^{N-1}\dbtilde x_i- \sum\limits_{j=N+1}^{M}\dbtilde y_i$ is a lifting of $x$ such that to $U$ such that the corresponding sequence of degrees of the summands $x_1,\ldots,x_{N-1,-y_{N+1},\ldots,-y_{M}}$ is $(k_1,\ldots,k_{N-1},k_{N+1},\ldots,k_M)$, hence lexicographically smaller than $(k_1,\ldots,k_N)$. This is a contradiction.

\end{proof}

Suppose additionally that $\gr_1 A $ is a commutative algebra. Then $\bigrr A$ is also commutative and has a structure of a Poisson algebra. Let $u$ and $v$ be homogeneous elements of the degrees $(i_1, j_1)$ and $(i_2, j_2)$ respectively. Let $\dbtilde u,\ \dbtilde v$ be their liftings to $F_1^{(i_1)}A \cap F_2^{(j_1)}A$ and $F_1^{(i_2)}A \cap F_2^{(j_2)}A$, respectively. Then we define 
\begin{eqnarray*}\{u, v \}: = [\dbtilde u,\dbtilde v] \mod (F_1^{(i_1+i_2-2)}A \cap F_2^{(j_1+j_2)}A + F_1^{(i_1+i_2-1)}A \cap F_2^{(j_1+j_2-1)}A),
\end{eqnarray*}
$$\deg \{u, v \}=(i_1+i_2, j_1+j_2).$$
Also $\gr_1 A$ is a Poisson algebra, and this give a Poisson algebra structure on $\bigrr A$. It follows from definitions that these brackets are the same Poisson brackets as on the associated bigraded algebra.

Also on $\gr_{12} A$ one can obtain a Poisson bracket from the commutator on $\gr_2 A$ which is also the same bracket as on the associated bigraded algebra. So we have the following

\begin{lem}
\label{2brackets}
If $gr_1 A$ is commutative then the associated bigraded algebra $\bigrr A$ is canonically isomorphic to $\gr_{12} A$ and $\gr_{21} A$ as Poisson algebra.
\end{lem}

Suppose that $\dim F_1^{(i)}A/F_1^{(i-1)}A$ is always finite and let $P_1(q):=\sum\limits_{i=0}^\infty q^i\dim F_1^{(i)}A/F_1^{(i-1)}A$ be the Poincar\'e series of $\gr_1 A$. Let $$P_{12}(q,z):=\sum\limits_{i,j=0}^\infty q^iz^j\dim \faktor{(F_1^{(i)}A \cap F_2^{(j)}A)}{(F_1^{(i-1)}A \cap F_2^{(j)}A + F_1^{(i)}A \cap F_2^{(j-1)}A)}$$
be the Poincar\'e series of $\bigrr A$. Then we have

\begin{lem}\label{le:Poincare}
$P_1(q)=P_{12}(q,1)$.
\end{lem}

\begin{prop}
We have a bigraded Poisson algebra isomorphism $\gr_{12} Y(\fg) \simeq \gr_{21} Y(\fg) \simeq S(\fg[t])$ where the bigrading on $S(\fg[t])$ is given by $\deg_1 x[r-1]=r,\ \deg_2 x[r-1]=r-1$ and the Poisson bracket on $S(\fg[t])$ is given by $\{ x[r],y[s]\}=[x,y][r+s]$. 
\end{prop}

\begin{proof}
The first isomorphism $\gr_{12} Y(\fg) \simeq \gr_{21} Y(\fg)$ is a particular case of Lemma~\ref{le:bigr}. 
From Theorem~\ref{th:F2} we have $x[r-1] \in F_1^{(r)}U(\fg[t])$. Hence the Poincar\'e series $P_{12}(q,z)$ is greater or equal to $\prod\limits_{r=1}^{\infty}(1-q^rz^{r-1})^{-\dim\fg}$ (in the sense that every coefficient of the former is greater or equal to the corresponding coefficient of the latter), and it is equal if and only if $x[r-1] \not\in F_1^{(r-1)}U(\fg[t])$ for all $x\ne0$. So according to Lemma~\ref{le:Poincare} $P_{1}(q)=\prod\limits_{r=1}^{\infty}(1-q^r)^{-\dim\fg}$ if and only if $x[r-1] \not\in F_1^{(r-1)}U(\fg[t])$ for all $x\ne0$. On the other hand, we have $P_{1}(q)=\prod\limits_{r=1}^{\infty}(1-q^r)^{-\dim\fg}$ by Corollary~\ref{th:F1}. This completes the proof.
\end{proof}

\subsection{Congruence subgroup $G_1[[t^{-1}]]$ and its coordinate ring.}\label{ss:congruence}
Let us give a few more details on $\gr_1 Y(\fg)=\mathcal{O}(G_1[[t^{-1}]])$. By definition, the proalgebraic group $G[[t^{-1}]]$ consist of $\bc[[t^{-1}]]$ points of $G$. For any $g\in G[[t^{-1}]]$ we denote by $ev_g$  the corresponding homomorphism $\mathcal{O}(G) \to \bc[[t^{-1}]]$. The first congruence subgroup $G_1[[t^{-1}]]\subset G[[t^{-1}]]$ is the kernel of the evaluation homomorphism at the infinity $G[[t^{-1}]] \to G$.

To any function $f \in \mathcal{O}(G)$ one can assign the $\bc[[t^{-1}]]$-valued function $\wt{f}: G_1[[t^{-1}]] \to \bc[[t^{-1}]], \wt{f} = \sum_{r=0}^{\infty} f^{(r)} t^{-r}$ as follows: for any $g\in G_1[[t^{-1}]]$ we have  
$$\wt{f}(g) = ev_g(f).$$
The Fourier coefficients $f^{(r)}$ for all $f\in \mathcal{O}(G)$ generate the coordinate ring $\mathcal{O}(G_1[[t^{-1}]])$. Note that the group $G_1[[t^{-1}]]$ depends only on the formal group scheme assigned to $G$, so one can produce $f^{(r)}$ from any $f$ in the completion of $\mathcal{O}(G)$ with respect to the maximal ideal of $E\in G$.

There is a natural Poisson bracket on $\mathcal{O}(G_1[[t^{-1}]])$ coming from the Lie bialgebra structure on the loop algebra $\fg((t^{-1}))$ (or, equivalently, from the rational $r$-matrix). To write this bracket explicitly, we set, for any $x\in\fg$, the corresponding momenta vector fields of the left and right action on $G$, $\xi^L_x$ and $\xi^R_x$, respectively. Then for $f_1,f_2\in\mathcal{O}(G)$, the bracket of corresponding $\bc[[t^{-1}]]$-valued functions reads
\begin{equation}
\{ \wt{f}_1(u),\wt{f}_2(v)\}=\frac{1}{u-v}( \wt{\xi^L_{x_a}f_1}(u)\wt{\xi^L_{x_a}f_2}(v)-\wt{\xi^R_{x_a}f_1}(u)\wt{\xi^R_{x_a}f_2}(v)).
\end{equation} 

The $\bc^*$ action on $G_1[[t^{-1}]]$ by dilations of the variable $t$ determines a grading on $\mathcal{O}(G_1[[t^{-1}]])$ such that $\deg f^{(r)}=r$ for any $f\in\mathcal{O}(G)$. The Poisson bracket has degree $-1$ with respect to this grading. A more precise statement of Theorem~\ref{th:F1} is the following

%\begin{thm}(\cite{kamn}, \cite{ir2})
%$\gr Y_{new}(\fg) \simeq \mathcal{O}(G_1[[t^{-1}]]) \simeq S(\fg[t]) \simeq \gr_1 Y_V(\fg)$ such that $\Delta_{\beta, v}^{(r)}$ corresponds to the image of $t_{\beta,v}^{(r)}$ in $r$-th filtered component. 
%\end{thm}
\begin{prop} \cite[Proposition 2.24]{ir2} 
\label{locprop}
There is an isomorphism of graded Poisson algebras $\gr_1 Y_V(\fg) \simeq \mathcal{O}(G_1[[t^{-1}]])$ such that $\gr_1 t_{ij}^{(r)} = \Delta_{ij}^{(r)}$ where $\Delta_{ij} \in \mathcal{O}(G)$ are the matrix elements of the representation $V$.
\end{prop}

Any formal diffeomorphism $\varphi: (\fg, 0) \to (G,E)$ (i.e. any regular map $\varphi$ from the formal neighborhood of $0\in\fg$ to the formal neighborhood of $E\in G$ such that $d_0\varphi=\Id_\fg$) determines an isomorphism $\Phi: t^{-1}\fg[[t^{-1}]]\to G_1[[t^{-1}]]$ which preserves the grading defined by the $\bc^*$ action by dilations on both sides. The coordinate ring of $t^{-1}\fg[[t^{-1}]]$ is the symmetric algebra of its graded dual space i.e. $\fg[t]$ with the pairing given by $$( x(t),y(t)):=Res_{t=0}\langle x(t),y(t) \rangle dt\quad\forall\ x(t)\in\fg[t],\ y(t)\in t^{-1}\fg[[t^{-1}]].$$ 
%Here $\left<\cdot,\cdot\right>$ is the invariant scalar product on $\fg$.
This means that any formal diffeomorphism $\varphi:(\fg,0)\to (G,E)$ identifies $\mathcal{O}(G_1[[t^{-1}]])$ with $S(\fg[t])$.  

The filtration $F_2$ on $Y(\fg)$ induces a filtration on $\gr_1 Y(\fg)=\mathcal{O}(G_1[[t^{-1}]])$. Slightly abusing notations, we denote this filtration by $F_2$ as well.
Let $\mathcal{O}(G)_+$ be polynomial functions on $G$ consisting of $f \in \mathcal{O}(G)$ such that $f(E) = 0$.

\begin{prop}\label{pr:F2}
For any $f\in \mathcal{O}(G)$, we have $f^{(r)}\in F_2^{(r-1)}\mathcal{O}(G_1[[t^{-1}]])$. Moreover, if $f=f_1 \cdot f_2\cdot\ldots\cdot f_k$ such that $f_1, \ldots, f_k \in \mathcal{O}(G)_+$ then $f^{(r)}\in F_2^{(r-k)}\mathcal{O}(G_1[[t^{-1}]])$.  
\end{prop}

\begin{proof}
It suffices to check the first assertion on generators of $\mathcal{O}(G)$. According to Peter-Weyl theorem, $\Delta_{ij}$ generate $\mathcal{O}(G)$ and we have $\Delta_{ij}^{(r)}=\gr_1 t_{ij}^{(r)}\in F_2^{(r-1)}\mathcal{O}(G_1[[t^{-1}]])$ by Proposition \ref{locprop}. 

To prove the second assertion, notice that $f^{(r)}$ is a linear combination of $f_1^{(r_1)}f_2^{(r_2)}\cdot\ldots\cdot f_k^{(r_k)}$ with $r_i>0$ and $r_1+r_2+\ldots+r_k=r$. 
\end{proof}

\begin{cor}\label{co:linear}
Let $f_a\in\mathcal{O}(G)$ be a collection of functions such that $\{x_a= d_E f_a\}$ is the basis of $\fg=\fg^*=T_E^*G$. Then $f_a^{(r)}\in F_2^{(r-1)}\mathcal{O}(G_1[[t^{-1}]])$, and $\gr_2 \mathcal{O}(G_1[[t^{-1}]])$ is freely generated by $\gr_2 f_a^{(r)}$ with $a = 1,\ldots,n,\ r=1,2,\ldots$. Moreover, $\gr_2 f_a^{(r)}=x_a[r-1]$. 
\end{cor}

\begin{cor}\label{co:coord-indep}
Let $\varphi_1:(\fg,0)\to (G,E)$ and $\varphi_2:(\fg,0)\to (G,E)$ be formal diffeomorphisms such that $d_0\varphi_1=d_0\varphi_2$. Let $\Phi_1^*,\Phi_2^*:\mathcal{O}(G_1[[t^{-1}]])\to S(\fg[t])$ be the corresponding ring isomorphisms. Then we have $\gr_2\Phi_1^*=\gr_2\Phi_2^*$.   
\end{cor}

This means that under \emph{any} identification $\mathcal{O}(G_1[[t^{-1}]])\simeq S(\fg[t])$ as above, the grading $F_1$ on $S(\fg[t])$ is given by $\deg x[r-1]= r$ and the filtration $F_2$ on $S(\fg[t])$ is given by $\deg x[r-1]= r-1$, for any $x\in\fg$. The Poisson bracket on $\mathcal{O}(G_1[[t^{-1}]])\simeq S(\fg[t])$ descends to $\gr_2 S(\fg[t])=S(\fg[t])$. We denote the latter bracket by $\{\cdot,\cdot\}_\ell$.

\begin{lem}
\label{le:poisson} We have
$$\{ x[m], y[l]\}_\ell = [x,y][n+m],$$
for any $x,y \in \fg, m, l \geq 0$.
\end{lem}
\begin{proof}
For any $x\in\fg$, denote by $\wt{x}(u)$ the formal series $\sum\limits_{r=1}^\infty x[r-1]u^{-r}\in S(\fg[t])[u^{-1}]$. Let $f_x\in\mathcal{O}(G)$ be a function such that $d_E f_x=x$ under the identification $\fg=\fg^*$. According to Corollary~\ref{co:linear}, we have  $f_x^{(r)}\in F_2^{(r-1)}\mathcal{O}(G_1[[t^{-1}]])$ and $\gr_2 f_x^{(r)}=x[r-1]$. Slightly abusing notations we will write $\gr_2\wt{f_x}(u)=\wt{x}(u)$.

For $x,y\in\fg$ we take the functions $f_x,f_y\in\mathcal{O}(G)$ as above and write the Poisson bracket $$\{\wt{f_x}(u),\wt{f_y}(v)\}=\frac{1}{u-v}(\sum\limits_{a=1}^{\dim \fg} \wt{\xi^L_{x_a}f_x}(u)\wt{\xi^L_{x_a}f_y}(v)-\wt{\xi^R_{x_a}f_x}(u)\wt{\xi^R_{x_a}f_y}(v)).$$
Taking the leading term of each coefficient in the expansion in the variables $u$ and $v$ on both sides of the equation, with respect to the filtration $F_2$, we get
$$\{\wt{x}(u),\wt{y}(v)\}_\ell=\frac{1}{u-v}(\sum\limits_{a=1}^{\dim \fg} \frac{1}{2}((x_a,x)[x_a,y(v)]+[x_a,x(u)](x_a,y)-(x_a,x)[y(v),x_a]-[x(u),x_a](x_a,y))).$$
Since $\sum\limits_{a=1}^{\dim \fg}((x_a,x)[x_a,y]=-\sum\limits_{a=1}^{\dim \fg}[x_a,x](x_a,y)=[x,y]$ we finally get
$$\{\wt{x}(u),\wt{y}(v)\}_\ell=[x(u),y(v)].
$$
\end{proof}

\subsection{}

From Proposition~\ref{2brackets} we obtain that
$\gr_{12} Y_V(\fg) \simeq \gr_{21} Y_V(\fg)$ as Poisson algebras and are isomorphic to $S(\fg[t])$ with the standard Kirillov-Kostant Poisson bracket.

We identify $\gr_{12} Y_V(\fg)$ with $S(\fg[t])$ and $\gr_{21} Y_V(\fg)$ with $S(\fg[t])$ thus obtain an automorphism of Poisson algebra
$\psi: S(\fg[t]) \to S(\fg[t])$.

\begin{lem}
\label{psi}
$\psi(x[k]) = c^kx[k], c \in \bc^*$.
\end{lem}
\begin{proof}
The filtrations $F_1,F_2$ are $\fg$-invariant, therefore $\psi$ is $\fg$-invariant. Let us identify $\gr_{12}$ and $\gr_{21}$ with bi-graded quotient.
We know that $$F_1^{(0)}Y(\fg) \cap F_2^{(0)}Y(\fg) = \bc \cdot 1,$$ $$\faktor{F_1^{(1)}Y(\fg) \cap F_2^{(0)}Y(\fg)}{F_1^{(0)}Y(\fg) \cap F_2^{(0)}Y(\fg)} = \fg,$$ $$\faktor{(F_1^{(2)}Y(\fg) \cap F_2^{(1)}Y(\fg))}{(F_1^{(1)}Y(\fg) \cap F_2^{(1)}Y(\fg) + F_1^{(2)}Y(\fg) \cap F_2^{(0)}Y(\fg))} = t \cdot \fg \simeq \fg$$ as $\fg$-modules. Note also that $$\faktor{F_1^{(1)}Y(\fg) \cap F_2^{(0)}Y(\fg)}{F_1^{(0)}Y(\fg) \cap F_2^{(0)}Y(\fg)} = \fg$$ is a Lie algebra isomorphism with respect to the Poisson bracket on the left hand side.

Using the fact that $\fg$ is simple we see that the only isomorphism of $\fg$ with itself is identity and isomorphism of $t \cdot \fg$ as $\fg$-module is the scalar of identity. 

From Lemma \ref{le:poisson} by induction on $r$ we have
$$\psi([x,y][r]) = \psi(\{x[r-1], y[1]\}) =  \{\psi(x[r-1]), c y[1]\} = \{c^{r-1} x[r-1], c y[1]\} = c^{r} [x,y][r].$$

Since $\fg$ is simple, we have $[\fg,\fg] = \fg$ therefore $\psi$ has the desired form.
\end{proof}

It will be useful for what follows to have a way  to write the leading term with respect to $F_2$ of any function of the form $f^{(r)}\in\mathcal{O}(G_1[[t^{-1}]])$ for any $f\in\mathcal{O}(G)$. For this, we fix a formal coordinate system in the neighborhood of $E\in G$, i.e. let $\varphi:(\fg,0)\to(G,E)$ be a formal diffeomorphism such that $d_0\varphi = \Id$. Then to any function $f\in\mathcal{O}(G)$ one can assign its Taylor expansion at $E\in G$, namely a collection of homogeneous polynomials $f_l\in S^l(\fg),\ l=0,1,\ldots$ such that $\varphi^*f=\sum\limits_{l=0}^{\infty} f_l$. 
We denote by $D$ the derivation on $S(\fg[t])$ determined by $$D(x[r-1])=rx[r].$$  

\begin{lem}\label{le:Taylor}
Suppose $f_k\in S^k(\fg)$ is the first nonzero term in the Taylor series of $f \in \mathcal{O}(\tilde G)_+$ at $E\in G$, $k>0$. 

Then we have \begin{enumerate} \item $f^{(r)}=0$ for $r<k$; \item $f^{(r)}\in F_2^{(r-k)}\mathcal{O}(G_1[[t^{-1}]])$ and $f^{(r)}\not\in F_2^{(r-k-1)}\mathcal{O}(G_1[[t^{-1}]])$ for $r\ge k$; \item $\gr_2 f^{(r)}=\frac{1}{(r-k)!}D^{r-k}f_k$ where $f_k\in S(\fg)\subset S(\fg[t])$ as a polynomial of the $x[0]$'s.
\end{enumerate}
\end{lem}

\begin{proof} 
The first assertion follows immediately from Proposition~\ref{pr:F2}. The second one follows from Corollary~\ref{co:linear}. To show the last equality, note that $\Phi^*\wt{f}(u)=\sum\limits_{l=k}^{\infty} \wt{f_l}(u)$. According to the assertions (1-2) and by Corollary~\ref{co:coord-indep} the leading term of any Fourier coefficient with respect to the filtration $F_2$ is given by that of $\Phi^*\wt{f_k}(u)$. On the other hand for any $x\in\fg$ the corresponding series $\wt{x}(u)=\sum\limits_{r=1}^\infty x[r-1]u^{-r}$ rewrites as $\wt{x}(u)=\exp (u^{-1}D) x[0]$. So we have $\Phi^*\wt{f_k}(u)=\exp (u^{-1}D) f_k$ as well, hence the assertion. 
\end{proof}

\section{Bethe subalgebras in Yangian}

\subsection{Definition.}
Let $\rho_i:Y_V(\fg)\to\End V(\omega_i,0)$ be the $i$-th fundamental representation of $Y(\fg)$. $V(\omega_i,0)$ is also a representation of $U(\fg)$ (because $U(\fg) \subset Y_V(\fg)$) and hence can be regarded as a representation of $\tilde G$. Slightly abusing notation we denote this group representation by the same symbol $\rho_i$.

Let $$\pi_i: V \to V(\omega_i,0)$$ be the projection. 

Let $T^i(u)=\pi_iT(u)\pi_i$ be the submatrix of $T(u)$-matrix, corresponding to the $i$-th fundamental representation.
\begin{defn}
\label{bethe}
Let $C \in \tilde G$. Bethe subalgebra $B(C)\subset Y_V(\fg)$ is the subalgebra generated by all coefficients of the following series  with the coefficients in $Y_V(\fg)$
$$\tau_i(u,C) = \tr_{V(\omega_i,0)} \rho_i(C)T^i(u),\ \  1 \leqslant i \leqslant n.$$
\end{defn}
\begin{rem}
\emph{In fact $B(C)$ depends only on the class of $C$ in $\tilde G / Z(\tilde G)$, i.e. on an element of adjoint group $G$.}
\end{rem}

\begin{prop}(\cite{ir2}, \cite{ilin})
\begin{enumerate}
\item Bethe subalgebra $B(C)$ is commutative for any $C \in G$.
\item $B(C)$ is a maximal commutative subalgebra of $Y(\fg)$ for $C \in T^{reg}$.
\end{enumerate}
\end{prop}

\begin{comment}
\subsection{Quadratic part of Bethe subalgebra.}
Here we use filtration $F_1$.

We want to compute the quadratic part of $$\tr_{V_{\omega_i}} \rho_i(C) (\rho_i \otimes 1) R(-u).$$

We know the degree two part of the universal $R$-matrix:
$$\sum_{\alpha \in \Phi^+} (J(x_{\alpha})\otimes x_{\alpha}^- - x_{\alpha} \otimes J(x_{\alpha}^-)) + \dfrac{1}{d_i}\sum_{i=1}^n (h_i \otimes J(t_{\omega_i}) - J(h_i)\otimes t_{\omega_i}) + \frac{1}{2} \Omega^2$$

Note that the only essential parts (which is not in $\fh + \fh^{(2)}$) is 
$$\dfrac{1}{d_j} \sum_j \rho_i(h_i) J(t_{\omega_j}) + \sum_{\alpha \in \Phi^+} (\rho_i(x_{\alpha} x_{\alpha^-}) \otimes x_{\alpha} x_{\alpha}^- + \rho_i(x_{\alpha}^- x_{\alpha}) x_{\alpha}^- x_{\alpha}).$$

By definition put \\
$$D = \begin{pmatrix}
\tr_{V_{\omega_i}} \frac{1}{d_j}\rho_i(C)\rho_i(h_j)
\end{pmatrix}_{i,j = 1, \ldots, n},$$ $$v = \begin{pmatrix}
\sum_{\alpha \in \Phi^+} \tr_{V_{\omega_1}} \rho_1(C) \rho_1([x_{\alpha}, x_{\alpha}^-]^+)x_{\alpha} x_{\alpha}^-  \\
\vdots \\
\sum_{\alpha \in \Phi^+} \tr_{V_{\omega_n}} \rho_n(C) \rho_n([x_{\alpha}, x_{\alpha}^-]^+)x_{\alpha} x_{\alpha}^- 
\end{pmatrix}.$$

Then 
$$\begin{pmatrix}
\tau_1(u,C)^{(2)} \\
\vdots \\
\tau_n(u,C)^{(2)}
\end{pmatrix} = 
D
\begin{pmatrix}
J(t_{\omega_1}) \\
\vdots \\
J(t_{\omega_n})
\end{pmatrix}
+ 
v$$
\end{comment}

\subsection{Bethe subalgebras in $\mathcal{O}(G_1[[t^{-1}]])$}
Here we follow \cite{ir2}.
Let $\{V_{\omega_i}\}_{i=1}^{n}$ be the set of all fundamental representations of $\fg$. We also consider  $\{V_{\omega_i}\}_{i=1}^{n}$ as a representations of the corresponding simply-connected group $\tilde G$.
Let $\Lambda_i$ be some basis of $V_{\omega_i}$. For any $v \in \Lambda_i$ we denote the corresponding element of dual basis by $v^* \in \Lambda_i^*$. By $\Delta_{v, v^*}\in \mathcal{O}(\tilde{G})$ we denote the corresponding matrix coefficient of $V_{\omega_i}$.

\begin{defn}
\label{def}
Let $C \in \tilde G$. Bethe subalgebra $\tilde B(C)$ of $\mathcal{O}(G_1[[t^{-1}]])$ is the subalgebra generated by of the coefficients of the following series:
$$\sigma_i(u,C) = \tr_{V_{\omega_i}} \rho_i(C)\rho_i(g) = \sum_{v \in \Lambda_i} \Delta_{v,v^{*}}(Cg) = \sum_{r=0}^{\infty} \sum_{v \in \Lambda_i} \Delta_{v,v^{*}}^{(r)}(Cg) u^{-r},$$
where $\Lambda_i$ is some basis of $V_{\omega_i}$, $g \in G_1[[t^{-1}]]$.
\end{defn}

\begin{rem}
This subalgebra depends only on the class of $C$ in $\tilde G / Z(\tilde G)$ as well.
\end{rem}

\begin{rem}
One can define the same subalgebra using all finite-dimensional representations of $\tilde G$.
\end{rem}

\begin{prop}(\cite{ir2})
\label{gr}
We have $\gr_1 B(C) = \tilde B(C)$ for any $C \in T^{reg}$.
\end{prop}
We generalize this Proposition \ref{gr} to any $C \in T$ below.

\subsection{Size of a Bethe subalgebra.}
Consider $B(C)$ with $C \in T$.
%Let $\chi$ be a regular semisimple element of $\mathfrak{z}_{\fg}(C)$. 
%Let $d_i(\chi)$ be a number of algebraically independent elements  of degree $i$ of $\mathcal{A}_{\chi}$.
%Let $k_i$ be the number of free generators of degree $i$ of $ZU(\mathfrak{z}_{\fg}(C))$. Note that $\sum_i k_i = \rk \fg$. 
In the next proposition we use the filtration $F_1$.

\begin{prop} (Lower bound for the size of Bethe subalgebra, see also \cite{ir})
\label{size}
Bethe subalgebra $B(C)$ contains $\rk \fg$ infinite series of algebraically independent elements such that every series consist of elements with the degrees $m_i+1, m_i +2, \ldots$, where $m_i$ are the exponents of $\fg$, $i = 1, \ldots, \rk \fg$.
%$\rk \fg- d_i(\chi) + k_i$ elements of degree $i$ for all $i\in\bz_{>0}$.
\end{prop}
\begin{proof}

Analogous to \cite[Proposition 4.8]{ir2} we have $\gr_1 B(C) \supset \tilde B(C)$.
We are going to find a set of algebraically independent elements in $\tilde B(C)$ of the same degrees as in Proposition statement with respect to the grading obtained from filtration $F_1$. 

Let $\sigma_i(u,C)$ be generators of Bethe subalgebra $\tilde B(C) \subset \mathcal{O}(G_1[[t^{-1}]])$.
One can extend $\sigma_i(u,C)$ to the group $G((t^{-1/2}))$ by means of Definition \ref{def}.  Denote by $\sigma_i(C)^{(r)}$ the coefficient of $u^{-r}$ in $\sigma_i(u, C)$. 
%In fact the required set consist of $\sigma_i(C)^{(r)}, i = 1, \ldots, \rk \fg, r \geq 1$. 

%Let $h \in \fg$ be an element from the $\fsl_2$ triple of the element $e$.
For any coweight $\nu$ of the maximal torus $T\subset G$, we denote by $t^\nu$ the corresponding $1$-parametric subgroup. Note that $t^\nu$ can be regarded as a $\mathbb{C}((t^{-1}))$-point of $T\subset G$, hence as an element of $G((t^{-1}))$. Consider also the element $t^{\tilde \rho} \in G((t^{-1/2}))$, $\tilde \rho = \sum_i \tilde \omega_i$, where $\tilde \omega_i$ are fundamental co-weights of $\mathfrak{z}_{\fg}(C)$. Note that this is well-defined element because $2 \tilde \rho$  belongs to co-weight lattice of $\fg$.

Let $e$ be the principal nilpotent element of the reductive algebra $\fz_{\fg}(C)$. The differential of $\sigma_i(C)^{(r)}$ at the point $\exp(e) \in G((t^{-1/2}))$ is naturally a linear functional on the tangent space $T_{\exp(e)} G((t^{-1/2})) \simeq \fg((t^{-1/2}))$. Hereinafter we identify $T_{g(t)}  G((t^{-1/2}))$ with $\fg((t^{-1/2}))$ by the left $G((t^{-1/2}))$-action for any $g(t) \in G((t^{-1/2}))$.

Then we have
$$d_{\exp(t^{-1}e)} \sigma_i(C)^{(r)} =  d_{t^{\tilde \rho}\cdot \exp(e)t^{-\tilde\rho}} \sigma_i(C)^{(r)} = (\Ad t^{\tilde \rho}) d_{\exp(e)} \sigma_i(C)^{(r)}.$$

The last equality follows from the invariance of $\sigma_i^{(r)}(C)$ under conjugation by $t^{\tilde \rho}$.

We are now consider the restriction of differentials to $T_E G_1[[t^{-1}]] \simeq t^{-1} \fg[[t^{-1}]]$.
Let $\chi_{\omega_i}$ be characters of $\tilde G$-modules $V_{\omega_i}, i = 1, \ldots, \rk \fg$.
%We restrict $d_{\exp(e)} \sigma(C)^{(r)}$ to $\mathfrak{z}_{\fg}(C)[[t^{-1}]].$ Note that one can decompose $V_{\omega_i} = V_i \oplus W$ as representation of $\mathfrak{z}_{\fg}(C)$. Here $V_i, i = 1, \ldots, \rk \fg$ are fundamental representations of $\mathfrak{z}_{\fg}(C)$ if $\mathfrak{z}_{\fg}(C)$ is a Levi subalgebra.

%Let $\tilde \omega_i$ be a set of fundamental weights of $\mathfrak{z}_{\fg}(C)$.

\begin{lem}
\label{gr2}
$d_{\exp(e)} \, \sigma_i(C)^{(r)} (xt^{-s}) = \delta_{r,s} d_{C^{-1}  \exp(e)} \chi_{\omega_i} (x)$ for any $x\in\fg$. 
\end{lem}
\begin{proof}
We have $$d_{\exp(e)} \sigma_i(C)^{(r)}=\sum\limits_{v\in\Lambda_i} d_{\exp(e)} \Delta_{v,v^{*}}^{(r)}(C) (xt^{-s}) = \delta_{r,s} \tr_{V_{\omega_i}} \rho_i(C) \rho_i(x)=\delta_{r,s} d_{C^{-1}  \exp(e)} \chi_{\omega_i} (x)$$ for any $x\in\fg$. 
\end{proof}

Note that $C^{-1} \cdot \exp(e)$ is a regular element of $\tilde G$. 
As in \cite{ir2} the key point here is the fact that differentials of characters of fundamental representations at regular point are linearly independent (see \cite[Theorem 3, p.119]{steinberg}).

\begin{lem}
\label{lem1}
$ \spann \left<d_{C^{-1} \exp(e)} \chi_{\omega_i}\right> = \fz_{\fz_{\fg}(C)}(e)$ under the identification $\fg^* \simeq \fg$.
\end{lem}
\begin{proof}
Note that $\fz_{\fg}(C \exp(e)) = \fz_{\fz_{\fg}(C)}(e)$.
It is sufficient to show now that $\spann \left<d_{C^{-1} \exp(e)} \chi_{\omega_i}\right> = \fz_{\fg}(C \exp(e))$. It is obvious that $\spann \left<d_{C^{-1} \exp(e)} \chi_{\omega_i}\right> \subset \fz_{\fg}(C \exp(e))$ and dimensions coincide according to linear independence of differentials at regular point.
\end{proof}

%Note that there is an $\fsl_2$-triple corresponding to $\tilde \rho$: $\{e, 2\tilde\rho, f\}$. 

Under the correspondence from previous Lemma  one can express eigenvector $v_j$ of $\tilde\rho$ with eigenvalue $m_j$ as a linear combination of $d_{C^{-1} \exp(e)} \chi_{\omega_i}$. Let $\sigma_{v_j}(C)^{(r)}$ be the corresponding linear combination of $\sigma_i(C)^{(r)}, i = 1, \ldots, \rk \fg$.
 %One can decompose $\fg = \bigoplus V_i$ -- eigenspaces of $2\rho$.

\begin{lem}
$(\Ad t^{\tilde \rho}) d_{\exp(e)} \, \sigma_{v_j}(C)^{(r)} (xt^{-s}) = \delta_{r,s - m_j}  \left<v_j, x \right>$ for any $x\in\fg$.
\end{lem}
\begin{proof}
It follows from Lemma \ref{gr2} and the fact that 
$\sigma_{v_j}^{(r)}(C)$ is an eigenvector of $t^{\tilde\rho}$ with eigenvalue $t^{m_j}$.
\end{proof}

From the last Lemma the statement of Proposition follows.
\end{proof}

\begin{rem}
We will also give another proof of Proposition \ref{size} in Section~5, see Proposition \ref{proof2}.
\end{rem}

%\section{Center on critical level and Gaudin model. }

\section{Universal Gaudin subalgebra }
%of $A \subset S(\fg[t])^{\fg}$
\subsection{Commutative subalgebra from the center on critical level.} We regard the Lie
algebra $\fg[t]$ as a ``half'' of
the corresponding affine Kac--Moody algebra $\hat\fg$ which is a central extension of the loop Lie algebra $\fg((t^{-1}))$. According to Feigin and Frenkel \cite{ff}, the local completion of the universal enveloping algebra $U(\hat{\fg})$ on the \emph{critical} level $k=-h^\vee$ has a huge center $Z$. The image of natural homomorphism from $Z$ to the quantum Hamiltonian reduction $(U(\hat{\fg})/U(\hat{\fg})t^{-1}\fg[t^{-1}])^{t^{-1}\fg[t^{-1}]}$ is a commutative subalgebra there. The latter naturally embeds into $U(\fg[t])$, so the image of $Z$ can be regarded as a commutative subalgebra $\mathcal{A}_{\fg}\subset U(\fg[t])$, which we call  the {\it universal Gaudin subalgebra} of $U(\fg[t])$. 

Though there are no explicit formulas for for the generators of $\mathcal{A}_{\fg}$ in general, one can describe explicitly the associated graded subalgebra $A_{\fg}\subset S(\fg[t])=\mathcal{O}(t^{-1}\fg[[t^{-1}]])$. Namely, $A_{\fg}$ is freely generated by all Fourier components of $\mathbb{C}[[t^{-1}]]$-valued functions $\Phi_l(x(t))$ for all generators $\Phi_l$ of the algebra of adjoint invariants $S(\fg)^\fg$. The subalgebra $A_{\fg}\subset S(\fg[t])$ can be obtained via Magri-Lenard scheme (\cite{ma}) from a pair of compatible Poisson brackets on $S(\fg[[t]])$ (see next subsection).

\subsection{Two Poisson brackets on $S(\fg[[t]])$.}

Let $\fg[[t]]$ be a Lie algebra of formal power series with coefficients in $\fg$.
Consider two Poisson brackets on $S(\fg[[t]])$:
$$\{x[n],x[m]\}_0 = [x,y][n+m];$$
$$\{x[n],y[m]\}_1 = [x,y][n+m+1],$$
for any $x,y \in \fg$.
Note that bracket $\{,\}_0$ is bracket we obtain on $\gr_{21} Y_V(\fg)$ if we restrict it to $S(\fg[t])$. Note also that $\fg[t]$ with $\{\cdot, \cdot\}_1$ is isomorphic to $t \cdot \fg[t]$ as a Lie algebra.

We call a pair of Poisson brackets on $S(\fg[[t]])$ {\it compatible} if every linear combination of them is a also a Poisson bracket. The following Lemma is well-known.

\begin{lem} \begin{enumerate} \item Poisson brackets $\{\cdot , \cdot\}_0$ and $\{\cdot , \cdot\}_1$ are compatible. \item Every linear combination of these brackets restricts to $S(\fg[[t]])^{\fg}$ (i.e. the bracket of $\fg$-invariant elements is $\fg$-invariant).
\end{enumerate}
\end{lem}

By $S(\fg[[t]])_{u,v}$ we denote a Poisson algebra $S(\fg[[t]])$ with Poisson bracket $u \{, \}_0+ v\{,\}_1$.
From the pair of compatible Poisson brackets one can obtain a Poisson commutative subalgebra of $S(\fg[[t]])^{\fg}$ with respect to $\{,\}_0$ and $\{,\}_1$ at the same time, see e.g. \cite{r1}. 
The construction is as follows:
subalgebra is generated by all centers of $S(\fg[[t]])_{u,v}^{\fg}$ for %general 
$u, v \in \bc$ except the case $u = 0, v = 1$.

%\begin{defn}
%Universal Gaudin subalgebra of $S(\fg[t])^{\fg}$ is subalgebra $A_{\fg} : = \mathcal{A}_{\fg} \cap S(\fg[t])^{\fg}$.
%\end{defn}

\subsection{Universal Gaudin subalgebra $A_{\fg}$.}
Consider the derivation $D$ of $S(\fg[t])^{\fg}$: 
$$D (x[n]) = (n+1) x[n+1].$$

Let $\Phi_i, i=1, \ldots, \rk \fg$ be free generators of $S(\fg[0])^{\fg}$. 

\begin{defn}
Universal Gaudin subalgebra $A_{\fg}$ is the subalgebra generated by all $D^k \Phi_i, k \geq 0, i = 1, \ldots, \rk \fg$. 
\end{defn}

\begin{prop}
Subalgebra $A_{\fg}$ is commutative and elements  $D^k \Phi_i, k \geq 0, i = 1, \ldots, \rk \fg$ are free generators of $A_{\fg}$.
\end{prop}
\begin{proof}

It is easy to check that the map 

$$\varphi_{1,v}: S(\fg[[t]])_{1,0} \to S(\fg[[t]])_{1,v};$$ $$x[m] \mapsto x[m] + \sum_{k=1}^{\infty} (-1)^k v^k x[m+k], $$
for all $x \in \fg, m \geq 0$, is an isomorphism of Poisson algebras. Indeed, the inverse map is 
$$x[m] \mapsto x[m] + v x[m+1],$$
for all $x \in \fg, m \geq 0$.
One can restrict $\varphi_{1,v}$ to $S(\fg[[t]])_{1,0}^{\fg}$ to obtain the isomorphism $S(\fg[[t]])_{1,0}^{\fg} \simeq S(\fg[[t]])_{1,v}^{\fg}$.

If $\Phi \in S(\fg[0])^{\fg}$ then it is central in $S(\fg[[t]])_{1,0}^{\fg}$ therefore $\varphi_{1,v}(\Phi)$ is central in $S(\fg[[t]])_{1,v}^{\fg}$. It implies that the elements of the form $\varphi_{1,v}(\Phi)$, $\Phi \in S(\fg[0])^{\fg}, v \in \bc$ commutes with respect to any bracket $u \{ \cdot, \cdot \}_0 + v \{ \cdot, \cdot \}_1$.
This implies that the coefficients of degrees of $v$ commute with respect to any bracket  $u \{ \cdot, \cdot \}_0 + v \{ \cdot, \cdot \}_1$, in particularly $\{\cdot, \cdot \}_0$ and $\{\cdot, \cdot \}_1$. Note that coefficients belongs to $S(\fg[t])^{\fg}$, hence these coefficients generate some commutative subalgebra of $S(\fg[t])^{\fg}$.

It is easy to see that for any $f \in S(\fg[0])$
$$\varphi_{1,v}(f) = \exp(-v D) f.$$

Therefore the coefficient of $v^k$ of $\varphi_{1,v}(\Phi)$ is proportional to $D^k \Phi$. This means that our subalgebra coincides with $A_{\fg}$.

The statement that elements of the form $D^k \Phi_i$ are free generators of $A_{\fg}$ is well-known, see e.g. \cite[\S 2.4.1]{bd}, \cite[Proposition 9.3]{f}.
\end{proof}
%It is known that these elements are free generators of $$.
%Note that the size of $A_{\fg}$ coincide with size of 

\subsection{Properties of subalgebra $A_{\fg}$}
By definition put
$$\omega_{\fg} = \sum_a x_a[0]^2 \in S(\fg[0])^{\fg},
$$
$$\Omega_{\fg} = \sum_{a} x_{a}[0] x_{a}[1] \in S(\fg[t])^{\fg},$$
where $\{x_a\}, a= 1, \ldots, \dim \fg$ is an orthonormal basis of $\fg$ with respect to $\langle\cdot,\cdot\rangle$.

Note that $\omega_{\fg} \in A_{\fg}$ by construction.
Also $ D \omega_{\fg} = 2 \Omega_{\fg}$ therefore $\Omega_{\fg} \in A_{\fg}$ too.

\begin{prop} (\cite{r1})
\label{centr2}
Subalgebra $A_{\fg}$ is the centralizer of $\omega_{\fg}$ in $S(\fg[t])$ with respect to $\{, \}_1$.
\end{prop}

\begin{prop}
\label{centr1}
Subalgebra $A_{\fg}$ is the centralizer of $\Omega_{\fg}$ in $S(\fg[t])^{\fg}$ with respect to $\{,\}_0$.
\end{prop}
\begin{proof}
%From Proposition \ref{centr2} we know that $A_{\fg}$ is a centralizer of $\omega_{\fg}$ in $S(\fg[t])$ with respect to $\{, \}_1$. Note that $\omega \in S(\fg[0])^{\fg}$. By $\fg[[t]]$ we denote the Lie algebra of formal power series with coefficients in $\fg$. We extend Poisson brackets $\{\cdot, \cdot\}_0$ and $\{\cdot, \cdot\}_1$ in an obvious way.
Let us again consider the isomorphism
$$\varphi_{1,v}: S(\fg[[t]])_{1,0} \to S(\fg[[t]])_{1,v}.$$

Note that $$\varphi_{1,v}(\omega_{\fg}) = \omega_{\fg} + \Omega_{\fg} v + \ldots$$ and $\varphi_{1,v}(\omega_{\fg})$ belong to the center of $S(\fg[[t]])_{1,v}^{\fg}$. Then for any $z \in S(\fg[t])^{\fg} \subset S(\fg[[t]])^{\fg}$ we have 

$$\{\varphi_{1,v}(\omega_{\fg}), z\}_0 + v \{\varphi_{1,v}(\omega_{\fg}), z\}_1 = 0. $$

Considering the coefficient of $v$ we get 
$$ \{\Omega_{\fg}, z\}_0 + \{\omega_{\fg}, z\}_1 = 0.$$
%Let us consider brackets $\{,\}_0 + \varepsilon \{,\}_1$.
%We note that $S(\fg[[t]])_{1,0} \simeq S(\fg[[t]])_{1,\varepsilon}$: $$x[m] \mapsto x[m] + \sum_{k=1}^{\infty} (-1)^k \varepsilon x[m+k].$$ We denote this isomorphism by $\varphi_{1,\varepsilon}$. %Consider an element $\omega = \sum_a x_a[0]^2 \in S(\fg[0])^{\fg}$.

%We also note that if $x \in S(\fg[0])^{\fg}$ then $\varphi_{1,\varepsilon}(x)$ belongs to the center of $S(\fg[[t]])_{1, \varepsilon}^{\fg}$. Indeed, every element of $S(\fg[[t]])_{1, \varepsilon}^{\fg}$ commutes with any $\varphi_{1,\varepsilon}(y[0]), y \in \fg$.

%Consider $\omega(\varepsilon) = \omega + C_1 \cdot \varepsilon + C_2 \cdot \varepsilon^2 + \ldots$ -- a central element of $S(\fg[t])_{1,\varepsilon}^{\fg}$. 
%We have

%$$\{\omega(\varepsilon), z(\varepsilon)\}_0 + \varepsilon \{\omega(\varepsilon), z(\varepsilon)\}_1 = 0, $$
%for any $z(\varepsilon) \in S(\fg[t])_{1,\varepsilon}^{\fg}$.

%Note that $C_1 = \Omega_{\fg}$.
%From the coefficient of $\varepsilon^1$ we obtain $$\{\Omega_{\fg}, x\}_0 + \{\omega_{\fg}, x\}_1 +  = 0$$ for any $x \in S(\fg[t])^{\fg}$.

Therefore the centralizer of $\omega_{\fg}$ in $S(\fg[t])^{\fg}$ with respect to $\{\cdot, \cdot\}_1$ coincides with the centralizer of $\Omega_{\fg}$ in $S(\fg[t])^{\fg}$ with respect to $\{\cdot, \cdot\}_0$.
\end{proof}
\begin{cor}
$A_{\fg}$ is a maximal commutative subalgebra of $S(\fg[t])_{1,0}^{\fg}$ and $S(\fg[t])_{0,1}^{\fg}$.
\end{cor}

\begin{prop}
\label{lift}
There exists no more than one lifting $\mathcal{A}_{\fg}$ of $A_{\fg}$ to $U(\fg[t])^{\fg}$.
\end{prop}
\begin{proof}
Up to scaling and additive constant there exists a unique lifting of $\Omega_{\fg}$ to $U(\fg[t])^{\fg}$.
Moreover, any lifting of subalgebra $A_{\fg}$ is the centralizer of the lifting of the element $\Omega_{\fg}$. But the centralizer does not depend on a constant therefore the lifting is unique. 
\end{proof}

\begin{rem}
\emph{We will assign to any $C \in T$ the subalgebra $A_{\fz_{\fg}(C)}\subset S(\fz_{\fg}(C))\subset S(\fg)$ and consider the elements $\omega_{\fz_{\fg}(C)}, \Omega_{\fz_{\fg}(C)}$ in it, i.e. consider the above objects for a reductive Lie algebra, not necessarily semisimple. All the statements and definitions of the present section remain the same for $\fz_{\fg}(C)$ with the following conventions: we take the restriction of $\langle \cdot, \cdot \rangle$ to $\fz_{\fg}(C)$ as the invariant scalar product on $\fz_{\fg}(C)$, $\rk \fz_{\fg}(C) = \rk \fg$, the exponents of $\fz_{\fg}(C)$ are the exponents of the semisimple algebra $[\fz_{\fg}(C), \fz_{\fg}(C)]$ plus additional $\rk \fg - \rk [\fz_{\fg}(C), \fz_{\fg}(C)]$ of zeros.}
\end{rem}

%Let $V$ be a non-trivial representation of the Yangian and consider $RTT$-realization $Y_V(\fg)$.
%Consider two flirtations $\deg t_{ij}^{(r)} = r$ and $\deg t_{ij}^{(r)} = r-1$. Denote by $\gr_1$ associated graded algebra with respect to the first filtration, by $\gr_2$ with respect to the second filtration. Then $\gr_1 Y_V(\fg) \simeq S(\fg[t])$, $\gr_2 Y_V(\fg) \simeq U(\fg[t])$.
%Also by PBW-theorem $\gr U(\fg[t]) \simeq S(\fg[t])$. We obtain two Poisson brackets on $S(\fg[t])$. Let $S(\fg[t])_1$ be a Poisson algebra from the first filtration, $S(\fg[t])_2$ from the second filtration.

%We define a grading on $\fg[t]$: $V_0 \subset V_1 \subset V_2 \subset \ldots $ and $V_k = \bigoplus_{i = 0}^{k}  \fg \cdot t^k$. Let $\gr_3$ be the associated graded algebra with respect to this grading. 

%\begin{thm}
%$\gr_3 S(\fg[t])_1 \simeq S(\fg[t])_2$ as Poisson algebras.
%\end{thm}

   %Under this identification we obtain a Poisson bracket on $S(\fg[t])$.

%as a commutative algebra, but we obtain different Poisson brackets. Before we have the following brackets:

%$$\{t_{\beta_1, v_1}(u),t_{\beta_2, v_2}(v)\} = \dfrac{1}{v-u} \sum_a \left(t_{\beta_1, x_a v_1}(u) t_{\beta_2, x^a v_2}(v) - t_{x_a \beta_1, v_1}(u) t_{x^a \beta_2, v_2}(v)\right)$$

%After:
%$$\{t_{\beta_1, v_1}^{(r)},t_{\beta_2, v_2}^{(s)}\} = \sum_a t_{x_a \beta_1, v_1}^{(0)} t_{x^a \beta_2, v_2}^{(r+s-1)} - \sum_a t_{\beta_1, x_a v_1}^{(0)} t_{\beta_2, x^a v_2}^{(r+s-1)}$$

\section{Bethe subalgebras and universal Gaudin subalgebras}
Let $E \in G$ be the identity element.  We are going to prove Theorem A for $C = E$.
%We first formulate Theorem A for $C = E$ because of two reasons: 1) The proof for $C = E$ is easier;
%2) Universal Gaudin subalgebra for $\fg$ we obtain from $B(E)$.
\begin{thm}
\label{result}
%Consider subalgebra $\mathcal{B}(E) = \gr_2 B(E) \subset U(\fg[t])$. 
$\gr_2 B(E)$ is the universal Gaudin subalgebra, i.e. $\gr_2 B(E) = \mathcal{A}_{\fg}$.
\end{thm}
\begin{proof}

\begin{lem}
\label{bigr}

The element $\Omega_{\fg}$ belongs to $\gr_{21} B(E)$ and to $\gr_{12} B(E)$.
 \end{lem}
\begin{proof}
%Firstly, we note that $\gr_{12} B(E) = \gr_{21} B(E)$ belongs to $S(\fg[t])^{\fg}$. 
Firstly we consider $\gr_{21} B(E) \subset S(\fg[t])^{\fg}$. From Proposition \ref{size} we know that there are 2 algebraically independent elements of degree 3 for type $A$ and 1 element of degree 3 for other types in $\gr_{21} B(E)$. All degree 3 elements in $S(\fg[t])^{\fg}$ are from $S^3(\fg)^{\fg} +  (\fg \cdot t \fg)^{\fg}$. In type $A$  the spaces $S^3(\fg)^{\fg}$ and $(\fg \cdot t \fg)^{\fg}$ are $1$-dimensional. In other types we have $S^3(\fg)^{\fg} = 0$ and $\dim (\fg \cdot t \fg)^{\fg}=1$. So $\gr_{21} B(E)$ contains the spaces $S^3(\fg)^{\fg}$ and $(\fg \cdot t \fg)^{\fg}$.

Any element from $(\fg \cdot t \fg)^{\fg}$ has the form $c \cdot \Omega_{\fg}, c \in \bc^*$.
Finally, $\Omega_{\fg}$ is homogeneous with respect to the second grading hence $\Omega_{\fg}\in\gr_{21} B(E)$. 

We identify $\gr_{12} Y_V(\fg)$ and $\gr_{21} Y_V(\fg)$ with $S(\fg[t])$ then obtain the automorphism $\psi: S(\fg[t]) \to S(\fg[t])$ of Poisson algebra.
From Lemma \ref{psi} it follows that $\psi$ maps any graded (with respect to the first grading) vector subspace of $S(\fg[t])$ to itself.
%Note that from Proposition \ref{size} it follows that there exists $\tilde \Omega_{\fg} \in \gr_1 B(E)$ such that $\gr_1 \tilde \Omega_{\fg} = \Omega_{\fg}$ and $\tilde \Omega_{\fg}$ is a linear combination of $\sigma_i^{(3)}(E)$.
%It follows that there exists $\dbtilde \Omega_{\fg} \in Y_V(\fg)$ such that $\gr_2  \dbtilde \Omega_{\fg} = \tilde \Omega_{\fg}$ and $\dbtilde \Omega_{\fg}$ is a linear combination of $\tau_i^{(3)}(E)$. 

We have $F_1^{(k)}Y_V(\fg)\cap F_2^{(l)}Y_V(\fg)=F_1^{(k)}Y_V(\fg)\cap  F_2^{(k-1)}Y_V(\fg)$ for $l\ge k$. Then any lifting of $\Omega_{\fg}$ to $Y_V(\fg)$ belongs to 
$$F_1^{(3)} \cap F_2^{(1)} + \sum_{k< 3, l < k } F_1^{(k)}  \cap F_2^{(l)} = F_1^{(3)} \cap F_2^{(1)}$$

Then for any lifting $\dbtilde \Omega_{\fg}$ we have $\gr_{12} \dbtilde \Omega_{\fg} = \Omega_{\fg}$.
\end{proof}

\begin{prop}
$\gr_{12} B(E) = A_{\fg}$.
\end{prop}
\begin{proof}
We know that $\Omega_{\fg} \in \gr_{21} B(E)$ and that $A_{\fg}$ is the centralizer of $\Omega_{\fg}$. Moreover $\gr_{21} B(E) \subset S(\fg[t])^{\fg}$ thus $\gr_{21} B(E) \subset A_{\fg}$. 

%It follows that $\gr_{12} B(E) = A_{\fg}$ because it is subset of $A_{\fg}$ and with the same Poincare series.
From Proposition \ref{size} and the definition of $A_{\fg}$ we see that Poincar\'e series of $\gr_{21} B(E)$ and of $A_{\fg}$ coincide. Hence we have $\gr_{21} B(E) = A_{\fg}$. Moreover, $\gr_{12} B(E) \subset \gr_{21} B(E)$, because $\Omega_{\fg} \in \gr_{12} B(E)$ and $\gr_{12} B(E)$ is Poisson commutative. Then by Proposition \ref{subgr} we have $\gr_{12} B(E) = \gr_{21} B(E) = A_{\fg}$.
\end{proof}

From the last proposition and Proposition \ref{lift} it follows that $\gr_{2} B(E) = \mathcal{A}_{\fg}$ and we are done.

\end{proof}

\begin{rem}
From Theorem \ref{result} we get the construction of $A_{\fg}$ independent of center on critical level of $\hat \fg$.
\end{rem}

\begin{cor}
\label{max1}
$B(E)$ is a maximal commutative subalgebra of $Y(\fg)^{\fg}$.
\end{cor}

\subsection{Application to the Gaudin model.}
%Let $U(\fg)^{\otimes n}$ be the tensor product of $n$ copies of
%$U(\fg)$. We denote the subspace $1\otimes\dots\otimes
%1\otimes\fg\otimes 1\otimes\dots\otimes 1\subset U(\fg)^{\otimes
%n}$, where $\fg$ stands at the $i$th place, by $\fg^{(i)}$.
%Respectively, for any $x\in U(\fg)$ we set
%\begin{equation}
%x^{(i)}=1\otimes\dots\otimes 1\otimes x\otimes
%1\otimes\dots\otimes 1\in U(\fg)^{\otimes n}.
%\end{equation}

%Let $\Delta_{[1,n]}:U(\fg[t])\hookrightarrow U(\fg[t])^{\otimes n}$ be the
%diagonal embedding (i.e. for $x\in\fg[t]$, we have $\Delta_{[1,n]}(x)=\sum\limits_{i=1}^nx^{(i)}$). To any nonzero $z\in\bc$, we assign the homomorphism $ev_z: U(\fg[t])\to U(\fg)$ of evaluation at the point $z$ (i.e., for $g\in\fg$, we have $\phi_z(g\otimes
%t^m)=z^mg$). For any collection of pairwise distinct nonzero
%complex numbers $z_i, i=1,\dots,n$, we have the following
%homomorphism:
%\begin{equation}
%\phi_{z_1,\dots,z_n}=(\phi_{z_1}\otimes\dots\otimes\phi_{z_n})\circ
%\Delta_{[1,n]}:U(\fg[t])\to U(\fg)^{\otimes n}.
%\end{equation}

%We define the  

Let $z \in \bc$.
Let $ev_z: U(\fg[t]) \to U(\fg)$ be an evaluation map.
Let $z_i, i=1, \ldots, n$ be different complex numbers. 
Let $d$ be the diagonal embedding
$$d: U(\fg[t]) \to U(\fg[t])^{\otimes n}.$$
We define a map $$ev_{z_1, \ldots, z_n}: = ev_{z_1} \otimes \ldots \otimes ev_{z_n} \circ d: U(\fg[t]) \to U(\fg) \otimes \ldots \otimes U(\fg).$$

\begin{defn}
Gaudin subalgebra $\mathcal{A}(z_1,\ldots,z_n)\subset U(\fg)^{\otimes n}$ is  $ev_{z_1,\dots,z_n}(\mathcal{A}_{\fg})$. 
\end{defn}

It is known (see \cite{ffr}) that this subalgebra is commutative and gives the complete set of integrals for the quantum Gaudin magnet chain.

From Theorem \ref{result} we have the following 
\begin{cor}
\label{corol2}
$ev_{z_1, \ldots, z_n}(\gr_2 B(E)) =  \mathcal{A}(z_1, \ldots, z_n)$.
\end{cor}

\begin{rem} \emph{Corollary \ref{corol2} generalizes Talalaev's formulas \cite{talalaev} for higher Gaudin Hamiltonians to $\fg$ of arbitrary type modulo knowledge of the universal $R$-matrix  $\hat R(u)$ for the Yangian. Namely, suppose that we know the expression of the universal $R$-matrix of Yangian in $PBW$-generators with respect to the filtration $F_2$ (i.e. ordered monomials form a basis of $Y(\fg)$ and their leading terms form a basis of $U(\fg[t])$, for example root generators in the ``new'' realization, see \cite{lev}). Then the leading terms of the generators of $B(E)$, namely leading terms of $\tau_i(u,E) = \tr_{V_{\omega_i}} (\rho_{\omega_i} \otimes \Id) \hat R(u), i = 1, \ldots, \rk \fg$, are the generators of $\gr_2 B(E)$. In type A the standard RTT generators are PBW generators, and the entries of $ (\rho_{\omega_i} \otimes \Id) \hat R(u)$ are the quantum $i\times i$ - minors, so we can write explicit formulas for the generators of $\gr_2 B(E)$ -- and these are precisely Talalaev's formulas.}
\end{rem}

%\section{Bethe subalgebras and universal Gaudin subalgebras}

\subsection{Proof of Theorem A in the general case.}

%t is possible to generalize the statement of Theorem \ref{result}. 

Let $C \in T$ and consider $\gr_2 B(C) \subset U(\fg[t])) $. As before $\fz_{\fg}(C)$ is the infinitesimal centralizer of $C$. In this subsection we are going to prove Theorem A in the full generality.

\begin{thm}
\label{result1}
$\gr_2 B(C)$ is the universal Gaudin subalgebra in $U(\fz_{\fg}(C)[t])^{\fz_{\fg}(C)}$, i.e. $\gr_2 B(C) = \mathcal{A}_{\fz_{\fg}(C)}$.
\end{thm}

%To prove it we try to follow the proof of Theorem \ref{result}.

%The main part is to show that $\gr_{12} B(C) \subset S(\fz_{\fg}(C)[t])^{\fz_{\fg}(C)}$.

\begin{prop}
\label{proof2}
%Element $\Omega_{\fz_{\fg}(C)}$ 

All generators of $A_{\fz_\fg(C)}$ belong to $\gr_{21} B(C)$.
\end{prop}
\begin{proof}

The idea of the proof is as follows. Consider any linear combination of functions
$$\sigma_i(C)(g):=\Tr_{V_{\omega_i}} \rho_i(C) \rho_i(g), g \in \tilde G$$
and obtain from it the series of functions on $\mathcal{O}(G_1[[t^{-1}]])$. Then these functions by definition belong to $\tilde B(C) \subset \gr_{21} B(C)$. We will find all generators of $A_{\fz_{\fg}(C)}$ using this construction.

Recall that $\mathcal{O}(\tilde G)_+$ is the set of polynomial functions on $\tilde G$ consisting of $f \in \mathcal{O}(\tilde G)$ such that $f(E) = 0$.
According to Lemma~\ref{le:Taylor} it is sufficient to show that there are functions in $\mathcal{O}(\tilde G)^{\tilde G}_+$ such that first non-zero term in Taylor expansion at the point $C^{-1}$ are $\Phi_i, i = 1, \ldots, \rk \fg$, where $\Phi_i$ are free generators of $S(\fz_{\fg}(C)[0])^{\fz_{\fg}(C)}$. 
%$f^{(2)} = \omega_{\fz_{\fg}(C)}$ such that all corresponding $F^{(r)}$ belongs to $\gr_1 B(C)$. Indeed, then we have $\gr_2 F^{(3)} = D \omega_{\fz_{\fg}(C)} = 2\Omega_{\fz_{\fg}(C)} \in \gr_{12} B(C)$. 
Indeed, it is equivalent to find $f \in \mathcal{O}(\tilde G)$ with $\Phi_i$ being the first non-zero term in the Taylor expansion of $f$ at the unity, as a linear combination of the functions  $\sigma_i(C)$.

We identify $\fg $ with $\fg^*$ by means of Killing form, so we can identify the coordinate ring $\mathcal{O}(\fg)$ with $S(\fg)$.
Consider the decomposition $\fg = \fz_{\fg}(C) \oplus \fn$, where $\fn$ is  sum of eigenspaces corresponding to eigenvalues of $\Ad(C)$ not equal to 1. This decomposition is orthogonal with respect to Killing form and thus $\mathcal{O}(\fz_{\fg}(C))$ gets identified with $S(\fz_{\fg}(C))$. 
We choose a formal coordinate system in the neighborhood of $C^{-1}\in \tilde G$ with the help of the map:
$$\Psi: \fz_{\fg}(C) \oplus \fn \to G, (h,x) \mapsto \exp(-x) C^{-1}\exp(h) \exp(x) $$
The differential of $\Psi$ at $C^{-1}$  is $(\Ad C^{-1} - \Id) \oplus \Id$ hence is non-degenerate. 

Let $\wt S(\fg)$ be the completion of $S(\fg)$ with respect to the maximal ideal of $0\in\fg$.
Consider the pullback $\Psi^*: \mathcal{O}(\tilde G) \to \wt S (\fg)$ which takes any function on $\tilde G$ to its Taylor series at $C^{-1}$ in our coordinates.
Let $f \in \mathcal{O}(\tilde G)^{\tilde G}$ be a central function. Since $f$ is constant on conjugacy classes its Taylor expansion $\Psi^*(f)$ does not depend on coordinates along $\fn$. So $\Psi^*(\mathcal{O}(\tilde G)^{\tilde G})\subset \wt S(\fz_{\fg}(C))$.
Let $Z_{\tilde{G}}(C)$ be a centralizer of $C$ in $\tilde G$. It is well-known that Lie algebra of $Z_{\tilde G}(C)$ is $\fz_{\fg}(C)$.  The map $\Psi$ is $Z_{\tilde{G}}(C)$-equivariant with respect to the adjoint action on $\fg$ and the action by conjugation on $\tilde G$, so we have $$
\Psi^*:\mathcal{O}(\tilde G)^{\tilde G}\to \wt S(\fz_{\fg}(C))^{\fz_{\fg}(C)}.$$

%Consider the Taylor series expansion of $f$ at point $C^{-1}$ as an element of the completion $\wt S(\fg)^{\fz_{\fg}(C)}$. 
Let $J = S(\fz_{\fg}(C))^{\fz_{\fg}(C)}_+=\{f \in S(\fz_{\fg}(C))^{\fz_{\fg}(C)}\ |\ f(0) = 0 \}$. Let $e$ be the principal nilpotent of $\fz_{\fg}(C)$. According to Kostant~\cite{kostant}, the differential at $e$ gives the the isomorphism $d_e: J /J^2 \simeq \fz_{\fz_{\fg}(C)}(e)$. Consider the composite map
$$\Theta: \mathcal{O}(\tilde G)^{\tilde G}_+ \to \wt S(\fz_{\fg}(C))^{\fz_{\fg}(C)}_+\to \wt S(\fz_{\fg}(C))^{\fz_{\fg}(C)}_+ / (\wt S(\fz_{\fg}(C))^{\fz_{\fg}(C)}_+)^2  \simeq J /J^2 \to \fz_{\fz_{\fg}(C)}(e). $$
Here the first arrow is $\Psi^*$, second arrow is the projection and the last is taking the differential at $e$. The resulting map $\Theta$ is just taking the differential at $C^{-1}\exp(e)$. By Lemma~\ref{lem1} the map $\Theta$ is surjective.

Let $\{e,h,f \}$ be the corresponding $\fsl_2$-triple in $\fz_{\fg}(C)$. One can split the centralizer
 $\fz_{\fz_{\fg}(C)}(e)$ into the eigenspaces of the operator $\frac{1}{2}\ad h$:
 
$$\fz_{\fz_{\fg}(C)}(e) = \bigoplus_{i=1}^{\rk \fg} V_i.$$

Note that $\dim V_i$ is the number of algebraically independent generators of $S(\fz_{\fg}(C))^{\fz_{\fg}(C)}$ of degree $m_i + 1$, see \cite{kostant}.
For any $f\in\mathcal{O}(\tilde G)^{\tilde G}_+$ whose Taylor series expansion starts from the $k$-th term we have $$\Theta(f) \in \bigoplus_{m_i \geq k-1}^{\rk \fg} V_i.$$ 

By the surjectivity of $\Theta$ we have functions with Taylor series on $C^{-1}$ starting from $\Phi_i, i = 1, \ldots, \rk \fg$ where $\Phi_i, i = 1, \ldots, \rk \fg$ are free generators of $S(\fz_{\fg}(C))^{\fz_{\fg}(C)}$ and we are done.
\begin{comment}
There are three cases: \newline
1) $\fz_{\fg}(C) = \fz_{\fg}(x)$ for some $x \in \fg$, i.e. $\fz_{\fg}(C)$ is a Levi subalgebra; \newline
2) $\rk \fz_{\fg}(C) = \rk \fg$; \newline
3) $\fz_{\fg}(C)$ is not a Levi suablgebra but $\rk \fz_{\fg}(C) < \rk \fg$.

1) In this case $\gr_{12} B(C)$ lie in $S(\fz_{\fg}(C)[t])^{\fz_{\fg}(C)}$. Moreover, the dimension of second graded component of $\gr_{12} B(C)$ coincide with that of $S(\fz_{\fg}(C)[t])^{\fz_{\fg(C)}}$. Therefore these two subspaces coincide and we obtain $\omega_{\fz_{\fg}(C)}$ in $\gr_1 B(C)$. 

2) Let $V$ be an arbitrary representation of $\fg$. Let $\{\lambda_i\}$ be a set of eigenvalues of $C$. Consider decomposition of $V$ into sum of eigenspaces of $C$:

$$V = \bigoplus_i V_{\lambda_i}.$$

We claim that 
$\Tr_{V} \rho(C)\rho([x_{\alpha}, x_{\alpha}^-]^+) = 0$ for any $\alpha \in \Phi \setminus \Phi_C$.
Indeed, on any $V_{\lambda_i}$ element $C$ acts by scalar therefore it is enough to show that $$\Tr_{V_{\lambda_i}} \rho([x_{\alpha}, x_{\alpha}^-]^+) = 0.$$ From the fact that rank of centralizer of $C$ coincide with rank of $\fg$ ot follows that one can express any root as a linear combination of roots from $\Phi_C$: $$\alpha = \sum_{\beta \in \Phi_C} a_{\beta} \beta,$$
where $a_{\beta} = 0,1$.

3) This case can be reduce to cases 1 and 2 by considering $\fz_{\fg}(C)$ as a subalgebra of some Levi subalgebra of the same rank.

In all three cases matrix $D$ is zero.

\end{comment}
\end{proof}

Proposition~\ref{proof2} implies that $A_{\fz_{\fg}(C)} \subset \gr_{21} B(C)$. To prove that in fact we have an equality we are going to prove that $A_{\fz_{\fg}(C)}$ is a maximal commutative subalgebra of $S(\fg[t])^{\fz_\fg(C)}$.

We mostly follow the argument of \cite{r1} below. Let $\{e,h,f\}$ be a principal $\fsl_2$-triple of $\fg$.
Let us recall two classical facts.

\begin{prop} (\cite{kostant})
\label{kostant}
Let $\pi$ be the restriction homomorphism
$$\pi: \bc[\fg] \to \bc[f + \fz_{\fg}(e)].$$
If we restrict $\pi$ to $\bc[\fg]^{\fg}$ we obtain an isomorphism $\bc[\fg]^{\fg} \simeq \bc[f + \fz_{\fg}(e)]$.
\end{prop}
The next proposition is well-known.
\begin{prop}
\label{algebraic}
Let $\psi$ be the restriction homomorphism
$$\psi: \bc[\fg] \to \bc[\fh].$$
If we restrict $\psi$ to $\bc[\fg]^{\fg}$ we obtain an isomorphism $\bc[\fg]^{\fg} \simeq \bc[\fh]^W$, where $W$ is the Weyl group of $\fg$. Particularly, $\bc[\fh]$ is an algebraic extension of $\psi(\bc[\fg]^{\fg})$. 
\end{prop}
%The problem now is that in contrast with $\gr_{21} B(E)$ where it is obvious that $\gr_{21} B(E) \subset S(\fg[t])^{\fg}$ we are only able to state that $\gr_{21} B(C) \subset S(\fg[t])^{\fz_{\fg}(C)}$ , but in is not obvious that $\gr_{21} B(C) \subset S(\fz_{\fg}(C)[t])^{\fz_{\fg}(C)}$. It follows from the following Proposition:

\begin{prop}
\label{centr}
$A_{\fz_{\fg}(C)}$ is a maximal commutative subalgebra of $S(\fg[t])^{\fz_{\fg}(C)}$ with respect to $\{\cdot, \cdot\}_0$. Moreover, $A_{\fz_{\fg}(C)}$ is centralizer of the element $\Omega_{\fz_{\fg}(C)}$.
\end{prop}
\begin{proof}
Let $$\psi: S(\fg[t]) \to S(\fh[t])$$ be a $\fh$-invariant projection.

%There are three cases:

%1) If $\fz_{\fg}(C)$ is a Levi subalgebra, i.e. a centralizer of some element of Lie algebra. Then it commutes with \newline
    %2) $\fz_{\fg}(C)$ is not a Levi suablgebra but $\rk \fz_{\fg}(C) < \rk \fg$. Then \newline
%3) $\rk \fz_{\fg}(C) = \rk \fg$.

%Let us consider last case.

%Let $\tilde A$ be the centralizer of $\Omega$ in $S(\fg[t])^{\fz_{\fg}(C)}$.
%We have $A \subset \tilde A$. We want to show that $A = \tilde A$. 
Let $\fg = \fz_{\fz_\fg(C)}(e) \oplus \fn$, where $\fn$ is any complement subspace. Consider $$\pi: S(\fg[t]) \to S(\fz_{\fz_\fg(C)}(e)[t])$$ such that $$\pi(x[n]) = x[n], x \in \fz_{\fz_\fg(C)}(e), \pi(x[n]) = \delta_{0n} \left<x,f\right> , x \in \fn.$$
This map generalize the map from Proposition \ref{kostant}.

\begin{lem} 
\label{514}
\begin{enumerate}
\item $\pi(A_{\fz_{\fg}(C)}) \simeq S(\fz_{\fz_{\fg}(C)}(e)[t])$; 
\item $A_{\fz_{\fg}(C)}$ is algebraically closed in  $S(\fg[t])$; 
\item $\psi(A_{\fz_{\fg}(C)}) \subset S(\fh[t])$ is an algebraic extension.
\end{enumerate}
\end{lem}
\begin{proof}
1) Let $\Phi_i, i = 1, \ldots, \rk \fg$ be  algebraically independent generators of $S(\fz_{\fg}(C)[0])^{\fz_{\fg}(C)}$.

We have $\pi(\bc[\Phi_1, \ldots, \Phi_k]) \simeq S(\fz_{\fz_\fg(C)}(e)[0])$ by Proposition \ref{kostant}. Moreover, $D$ commutes with $\pi$ hence we have $$\pi(\bc[D^s \Phi_1, \ldots, D^s \Phi_k]) \simeq S(\fz_{\fz_\fg(C)}(e)[s])$$ and hence $$\pi(A_{\fz_{\fg}(C)}) \simeq S(\fz_{\fz_\fg(C)}(e)[t]).$$

2) Suppose that $A_{\fz_{\fg}(C)}$ is not algebraically closed. Let $a \in S(\fg[t])$ be an element which is algebraic over $A_{\fz_{\fg}(C)}$. Then by the first statement of this Lemma we can assume that $\pi(a) = 0$.
Suppose that $p_n a^n + \ldots + p_1 a + p_0 = 0$, where $p_i \in A_{\fz_{\fg}(C)}$ and $n$ is minimal. Then $\pi(p_0) = 0$, so we have $(p_n a^{n-1} + \ldots + p_1) a = 0$. But $S(\fg[t])$ does not have zero divisors and we have contradiction with minimality of $n$.

3) Note that $\psi(\bc[\Phi_1, \ldots, \Phi_{\rk \fg}]) \subset S(\fh[0])$ is the algebraic extension from Proposition \ref{algebraic}. Using the fact that $D$ commutes with $\psi$ we see that $$\psi(\bc[\Phi_1, \ldots, \Phi_{\rk \fg}, D \Phi_1, \ldots, D \Phi_{\rk \fg}, \ldots,  D^s \Phi_1, \ldots, D^s \Phi_{\rk \fg}]) \subset  S(\bigoplus_{i=0}^s\fh[i])$$ is the algebraic extension as well for any $s$.

\end{proof}

%\begin{lem}
%Centralizer of $Z(\fz_{\fg}(C)[1])$ and $\Omega_{\fz_{\fg}(C)}$ in $S(\fg[t])^{\fz_{\fg}(C)}$ with respect to $\{\cdot, \cdot\}_1$ coincide with centralizer of $Z(\fz_{\fg}(C)[0])$ and $\omega_{\fz_{\fg}(C)}$ in $S(\fg[t])^{\fz_{\fg}(C)}$ with respect to $\{\cdot , \cdot\}_0$.
%\end{lem}
%\begin{proof}
%Here we just repeat the proof of Proposition \ref{centr1} separately for any element from linear part, and separately for elements $\Omega_{\fz_{\fg}(C)}$ and $\omega_{\fz_{\fg}(C)}$.
%\end{proof}
From Lemma \ref{centr2} it follows that the centralizer of $\omega_{\fz_{\fg}(C)}$ in $S(\fg[t])_{0,1}^{\fz_{\fg}(C)}$ contains the subalgebra $A_{\fz_{\fg}(C)}$.

Now let us define the family of automorphisms of $S(\fg[t])$ with respect to the bracket $\{\cdot , \cdot \}_1$. Let $\varphi_s(x[m]) = x[m] + s \delta_{0m} \left<h,x\right>$. It is a straightforward computation that $\varphi_s$ is an automorphism. We use the notion of a limit subalgebra in the next Lemma (see subsection \ref{limitgraded} for the definition). We also use the notion of a limit of a one-parametric family of subalgebras analogous to that of \ref{limit}.

\begin{lem} We have $h[0] \in \lim_{s \to \infty} \varphi_s(A_{\fz_{\fg}(C)})$.
\end{lem}

\begin{proof}
%Let $\tilde e$ be a principal nilpotent of $\fz_{\fg}(C)$ and let $\{\tilde e, \tilde h,\tilde f\}$ be the corresponding $\fsl_2$-triple.
Recall that an element $h$ is an element from principal $\fsl_2$-triple of $\fg$.
It is straightforward computation that $$\lim_{s \to \infty} \dfrac{\varphi_s(\omega_{\fz_{\fg}(C)}) - s^2 \langle h,h\rangle}{2s} = h[0].$$
%where $\{ t_a \}, a = 1, \ldots, \rk \fg$ is an orthonormal basis of $\fh$.

%Note just now we can just add to $\tilde h[0]$ an element from $Z(\fz_{\fg}(C)[0])$ and obtain $h[0]$ in $\varphi_s(A_{\fz_{\fg}(C)})$.
\end{proof}

\begin{lem}(\cite{r1})
The algebra $S(\fh[t])$ is the centralizer of $h[0]$ in $S(\fg[t])$ with respect to $\{ \cdot, \cdot\}_1$. 
\end{lem}

Now return to the proof of the Proposition. The centralizer of $\Omega_{\fz_{\fg}(C)}$ with respect to $\{\cdot, \cdot\}_0$ coincide with the centralizer of $\omega_{\fz_{\fg}(C)}$ with respect to $\{\cdot, \cdot\}_1$. Suppose that we have some element $a\not\in A_{\fz_{\fg}(C)}$ in centralizer of $\omega_{\fz_{\fg}(C)}$ with respect to $\{\cdot, \cdot\}_1$. From Lemma \ref{514} it follows that $a$ should be transcendental over $A_{\fz_{\fg}(C)}$ so we can assume that $\psi(a) = 0$. For some $k$ we have a non-zero limit $\tilde a = \lim_{s \to \infty} \frac{\varphi_s(a)}{s^k} \in \varphi_s (A_{\fz_{\fg}(C)})$. This limit should lie in the centralizer of the element $h[0]$, and then lie in $S(\fh[t])$. It means that $\psi(a) \ne 0$ which is a contradiction and completes the proof.

\end{proof}

\begin{cor}
\label{corol3}
We have $\gr_{12} B(C) = A_{\fz_{\fg}(C)}$ and $\gr_2 B(C) = \mathcal{A}_{\fz_{\fg}(C)}$.
\end{cor}
\begin{proof}
Analogous to Lemma \ref{bigr} we can find $\Omega_{\fz_{\fg}(C)}$ in $\gr_{12} B(C)$. Then from Proposition \ref{centr} and the fact that $\gr_{12} B(C)$ is Poisson commutative it follows that $\gr_{12} B(C) \subset \gr_{21} B(C)$ and then by Proposition \ref{subgr} we have $\gr_{12} B(C) \subset \gr_{21} B(C) = A_{\fz_{\fg}(C)}$.
From the uniqueness of lifting of $\Omega_{\fz_{\fg}(C)}$ to $U(\fg[t])$ (up to scalar and additive constant) we see that $\gr_2 B(C) = \mathcal{A}_{\fz_{\fg}(C)}$.
\end{proof}

This finishes the proof of Theorem~\ref{result1}.

\begin{cor}\label{maxc}
$B(C)$ is a maximal commutative subalgebra of $Y(\fg)^{\mathfrak{z}_{\fg}(C)}$ and $\tilde B(C)$ is a maximal Poisson commutative subalgebra in $\mathcal{O}(G_1[[t^{-1}]])^{\mathfrak{z}_{\fg}(C)}$.
\end{cor}
\begin{proof}
This follows from Proposition~\ref{centr} and Corollary~\ref{corol3}, since if the associated graded of a commutative subalgebra is maximal Poisson-commutative then the original subalgebra is maximal commutative as well.
\end{proof}

\begin{cor}
$\gr_1 B(C) = \tilde B(C)$ for any $C \in T$.
\end{cor}
\begin{proof}
This is immediate from Corollary~\ref{maxc} since $\gr_1 B(C) \supset \tilde B(C)$ and $\tilde B(C)$ is maximal Poisson-commutative.
\end{proof}

\section{Some limits of Bethe subalgebras.}
\subsection{Closure of the family of subspaces in a vector space}
Let $\{U(m)\}_{m \in M}, U(m) \subset \bc^n$ be a family of vector subspaces of the dimension $k$ parameterized by a complex algebraic variety $M$ (we allow $U(m_1) = U(m_2)$ for different points $m_1,m_2 \in M$).  Suppose that the resulting map $\theta: M \to {\rm Gr}(k,n)$ is regular. Consider the closure $Z$ of $\theta(M)$ in $Gr(k,n)$. The variety $Z$ naturally parameterizes a family of subspaces extending $\{U(m)\}_{m \in M}$. 
According to the general principle (see e.g. \cite{serre}) the closure of an algebraic variety under a regular map with respect to Zariski topology coincides with its closure with respect to the analytic topology. We will use this fact in Subsection \ref{limit}.
%Let $D^{\times} = \Spec \bc((t))$ and $D = \Spec \bc[[t]]$ and suppose that $m: D^{\times} \to M$ is a regular map.
%Then it is well-known that there exist unique regular map $m^{\prime}: D \to Z$ such that $m^{\prime}_{D^{\times}} = m$.
%By definition put  $\lim_{\varepsilon \to 0} U(m(\varepsilon)) = U(m^{\prime}(0))$. 

\subsection{Closure of the family of subspaces in a graded space}
\label{limitgraded}
Suppose that we have a graded vector space with finite-dimensional graded components $V = \bigoplus_i V_i$ such that $\dim V_i = n_i$ and a family of subspaces of the form $U(m) = \bigoplus_i U_i(m)$ with fixed dimensions $\dim U_i = k_i$ parametrized by a complex algebraic variety $M$, $m \in M$. Suppose that all maps $$\theta_r: M \to \prod_{i=1}^r {\rm Gr}(k_i,\ n_i), m \mapsto (U_1(m), \ldots, U_r(m))$$ are regular. Let $Z_r$ be the closure of $\theta_r(M)$.  There are well-defined projections  $\zeta_r: Z_r \to Z_{r-1}$ for all $r \geqslant 1$. The inverse limit $Z = \varprojlim Z_r$ is well-defined as a pro-algebraic scheme and is naturally a parameter space for some family of graded vector subspaces of fixed dimension. Indeed, any point $z\in Z$ is a sequence $\{z_r\}_{r \in \mathbb{N}}$ where $z_r \in Z_r$ and  $U(z) = \bigoplus_i U(z_i) \subset V$. We note also that Poincar\'e series of $U(z)$ coincide with Poincar\'e series of $U(m), m \in M$. 

%As before if we have $m: D^{\times} \to M$  then we are able to define $$\lim_{\varepsilon \to 0} U(m(\varepsilon)): =\bigoplus_i \lim_{\varepsilon \to 0} U_i(m(\varepsilon)).$$

Suppose that we have a structure of an algebra on $V$ and suppose that a family of graded subspaces is a family of subalgebras. Then any $U(z), z \in Z$ is also a subalgebra because being a subalgebra is a Zariski-closed condition on the parameters. If all subalgebras of our family are commutative, then any $U(z), z \in Z$ is commutative as well, by the same argument.
\subsection{Closure of the family of subspaces in a filtered space}
\label{limitfilt}
Suppose that we have a filtration on vector space $V$ with finite-dimensional filtered components:
$$V_0 \subset V_1 \subset V_2 \subset \ldots,$$
where $\dim V_i = n_i$.
Suppose that we have a family of subspaces of the form $$U(m) = \bigcup_i U_i(m), U_i(m) = U(m) \cap V_i$$ with fixed dimensions $\dim U_i(m) = k_i$ parameterized by a complex algebraic variety $M$, $m \in M$.  Suppose that all maps $\theta_r: M \to \prod_{i=1}^r {\rm Gr}(k_i,\ n_i)$ are regular. As in the case of graded vector subspaces, we have the pro-algebraic scheme $Z$. This scheme naturally parameterizes some family of subspaces in $V$. 

Indeed, every $z_r$ is a point in $\prod_{i=1}^r {\rm Gr}(k_i, n_i)$ i.e. a collection of subspaces $U_i^{(r)}(z) \subset V_i$ such that $U_i^{(r)}(z)\subset U_{i+1}^{(r)}(z)$ for all $i<r$. Since $\zeta_r(z_r) = z_{r-1}$ we have $U_i^{(r)}(z)= U_{i-1}^{(r)}(z)$ for all $i<r$. We define the subspace corresponding to $z\in Z$ as $U(z):= \bigcup_{r=1}^{\infty} U_r^{(r)}(z)$. 

%As before for any $m: D^{\times} \to M$ we define the limit 
%$$\lim_{\varepsilon \to 0} U(m(\varepsilon)) = \bigcup_{r=0}^{\infty} \lim_{\varepsilon \to 0} U_r^{(r)}(m(\varepsilon)).$$

In fact $Z$ parameterizes a bit richer data, namely, subspaces along with a filtration $\{U_i(z)\}$ such that $U_i(z) \subset U(z) \cap V_i$ (i.e. the inclusion is not necessarily equality since the dimension of $U(z) \cap V_i$ could jump).
We define the Poincar\'e series of a subspace as the Poincar\'e series of the associated graded space.
In contrast with the graded vector space the Poincar\'e series of $U(z)$ is not necessary the same -- but always not smaller lexicographically than that of $U(m), m \in M$.

Again, as in the case of graded vector space if $V$ is an algebra and all $U(m), m \in M$ are (commutative) subalgebras then any $U(z), z \in Z$ is a (commutative) subalgebra. 

\subsection{Definition of limit Bethe subalgebras.}
\label{limit}
Let $C$ be an element of $G^{reg}$ or $T^{reg}$. 
Recall that the formula $\deg t_{ij}^{(r)} = r$ defines the filtration $F_1$ on $Y_V(\fg)$. Recall that $F_1^{(r)}Y_V(\fg)$ is a $r$-th filtered component.
Consider $B^{(r)}(C):= F_1^{(r)}Y_V(\fg) \cap B(C)$.  In the paper \cite{ir2} (in the course of the proof of Theorem 2.6) it is proved that the images of the coefficients of  $\tau_1(u,C), \ldots,$ $\tau_n(u,C)$ freely generate the subalgebra $\gr B(C)\subset \gr_1 Y_V(\fg)$. Hence the dimension $d(r)$ of $B^{(r)}(C)$ does not depend on $C$.
Now we can apply results of subsection \ref{limitfilt} to this situation and obtain pro-algebraic schemes $Z_G$ and $Z_T$ which parameterize some families of commutative subalgebras.

\begin{prop} 
%(\cite[Proposition 4.11]{ir2})
For any $z \in Z$ subalgebra $B(z)$ is a commutative subalgebra of $Y_V(\fg)$. The Poincar\'e series of $B(z)$ is (lexicographically) not smaller that the Poincar\'e series series of $B(C)$ for $C\in G^{reg}$. We call a subalgebra of the form $B(z), z \in Z$ limit subalgebra.
\end{prop}

Following \cite{shuvalov} we want to explain a practical way to find some limit of one-parametric families of subalgebras and some elements belonging to them in the case of $Z_T$. 

By definition put $k_i = \dim B^{(i)}(C)$, $n_i = \dim F_1^{(i)} Y_V(\fg)$. Let $D$ be a small neighborhood of zero in $\bc$. Suppose that $c: D \to T$ is an analytic map and that the restriction  is a map $c^{\prime}: D \setminus \{0\} \to  T^{reg}$. Let $c_i^{\prime\prime} = \theta_i \circ c^{\prime}: D \to Gr(k_i,n_i)$. 
\begin{lem}
$c_i^{\prime\prime}$ (uniquely) extends to $D$ as an analytic map. 
\end{lem}
\begin{proof}
The map $c_i''$ is uniquely determined by the collection of Pl\"ucker coordinates of $c_i''(\varepsilon)$, which are Laurent series of $\varepsilon$ up to proportionality. So we can multiply all these coordinates by some power of $\varepsilon$ to make them Taylor series not all vanishing at the origin. This gives us the desired analytic map $D\to Gr(k_i,n_i)$. 
\end{proof}
We define $\lim\limits_{\varepsilon \to 0} B(c(\varepsilon)) : = \bigcup_i c_i''(0)$.
Now let us give a practical way to find some elements in $\lim_{\varepsilon \to 0} B(c(\varepsilon))$. Suppose that $d: D \to F_1^{(i)} Y_V(\fg)$ is an analytic map such that $d(\varepsilon) \in B(c(\varepsilon))$ for all $\varepsilon \ne 0$ and $d(\varepsilon) \ne 0$ for small enough $\varepsilon$. Let $n_i = \dim F_1^{(i)} Y_V(\fg)$. We obtain a map $d^{\prime}: D \setminus \{0\} \to \bp^{n_i -1}$. Analogous to the previous Lemma we know that there exists the unique continuation $d^{\prime\prime}: D \to \bp^{n_i -1}$ of $d^{\prime}$. Practically, this means that we can divide $d$ by some power of $\varepsilon$ to make it nonzero at the origin. It is natural to call $d''(0)$ the \emph{limit} of $d(\varepsilon)$ as $\varepsilon\to0$, i.e. $d''(0)=\lim\limits_{\varepsilon\to0}d(\varepsilon)$.

\begin{lem} We have
$\lim\limits_{\varepsilon\to0}d(\varepsilon) \subset \lim\limits_{\varepsilon \to 0} B(c(\varepsilon))$.
\end{lem}
\begin{proof}
Let $\diag:D\setminus\{0\} \to D\setminus\{0\} \times D\setminus\{0\}$ be the diagonal embedding.
The image of the map $c_i'\times d'\circ\diag:D\setminus\{0\}\to Gr(k_i,n_i)\times \bp^{n_i -1}$ belongs to the closed subset given by the incidence condition. Hence the same is true for $c_i''\times d'':D\to Gr(k_i,n_i)\times \bp^{n_i -1}$. So $d''(0)\in \lim_{\varepsilon \to 0} B(c(\varepsilon))$.
\end{proof}

\subsection{Shift of argument subalgebras.}
%добавить про конструкцию из центра на критю уровне
Let $\fg$ be a reductive Lie algebra.  We identify $\fg$ with $\fg^*$ by means of non-degenerate invariant scalar product. To any $\chi\in\fg^*=\fg$ one can assign a Poisson-commutative subalgebra in $S(\fg)$ with respect to the standard Poisson bracket (coming from the universal enveloping algebra $U(\fg)$ by the PBW theorem).
Let $ZS(\fg)=S(\fg)^{\fg}$ be the center of $S(\fg)$ with
respect to the Poisson bracket. The algebra $A_{\chi}\subset S(\fg)$
generated by the elements $\partial_{\chi}^n\Phi$, where $\Phi\in
ZS(\fg)$, (or, equivalently, generated by central elements of
$S(\fg)=\bc[\fg^*]$ shifted by $t\chi$ for all $t\in\bc$) is
Poisson-commutative and has maximal possible transcendence degree. More precisely, we have the following

\begin{thm}\label{mf} ~\cite{mf} For regular semisimple $\chi\in\fg$ the algebra $A_{\chi}$
is a free commutative subalgebra in $S(\fg)$ with
$\frac{1}{2}(\dim\fg+\rk\fg)$ generators (this means that
$A_{\chi}$ is a commutative subalgebra of maximal possible
transcendence degree). One can take the elements
$\partial_{\chi}^n\Phi_k$, $k=1,\dots,\rk\fg$,
$n=0,1,\dots,\deg\Phi_k-1$, where $\Phi_k$ are basic
$\fg$-invariants in $S(\fg)$, as free generators of $A_{\chi}$.
\end{thm}

\begin{thm} (\cite{ryb06})
For any regular semisimple $\chi \in \fg$ there exist a lifting $\mathcal{A}_{\chi} \subset U(\fg)$, i.e. a commutative subalgebra $\mathcal{A}_{\chi} \subset U(\fg)$ such that $\gr \mathcal{A}_{\chi} = A_{\chi}$. Moreover, this lifting is unique for generic regular $\chi$.
\end{thm}

\begin{thm} (\cite{ryb06})
For generic (i.e. belonging to some complement of countably many proper closed subsets in $\fh^{reg}$) regular semisimple $\chi \in \fg$, the subalgebra $\mathcal{A}_{\chi}$ is the centralizer of its quadratic part which is the linear span of the elements $\sum_{\alpha \in \Phi^+} \frac{(\alpha, h)}{(\alpha,\chi)} e_{\alpha} e_{-\alpha}$ for all $h\in\fh$.
\end{thm}

\subsection{Certain limits of Bethe subalgebras.}
Let $E \in G$ be the identity element.
\begin{thm}
Let $C(\varepsilon) = \exp(\varepsilon \chi), \chi \in \fh$ and $C(\varepsilon) \in T^{reg}$ if $\varepsilon \ne 0$.
Then $$\lim_{\varepsilon \to 0} B(C(\varepsilon)) = B(E) \otimes_{Z(U(\fg))} \mathcal{A}_{\chi}$$ for generic $\chi\in\fh^{reg}$.
\end{thm}

\begin{proof}
According to \cite{ilin} the quadratic part of Bethe subalgebra contains the following elements:

$$\sigma_i(C) = 2J(t_{\omega_i}) - \sum_{\alpha \in \Phi^+} \dfrac{e^\alpha(C)+1}{e^\alpha(C)-1} (\alpha, \alpha_i)x_{\alpha} x_{\alpha}^- \in Y(\fg), i = 1, \ldots, \rk \fg.$$
Here $J(t_{\omega_i})$ is an element of $Y(\fg)$ which does not depend on $C$. In the limit $\varepsilon \to 0$, the leading term has the form
$$\sum_{\alpha \in \Phi^+} \frac{2(\alpha, \alpha_i)}{(\alpha,\chi)} x_{\alpha} x_{\alpha}^- ,$$
i.e. the quadratic part of shift of argument subalgebra $\mathcal{A}_{\chi}$ (see \cite{vinberg}). As we state before for generic $\chi$ shift of argument subalgebra $\mathcal{A}_{\chi}$ is a centralizer of its quadratic part.

\begin{lem}\label{le:knop}
\begin{enumerate}
\item Suppose that $\fg$ is a reductive Lie algebra, $\fg_0$ -- a reductive subalgebra of $\fg$. Then the subalgebras $Y(\fg)^{\fg_0}$ and $U(\fg_0)$ in $Y(\fg)$ are both free $U(\fg_0)^{\fg_0}$-modules. Moreover, the product of these subalgebras in $Y(\fg)$ is:
$Y(\fg)^{\fg_0} \cdot U(\fg_0) \simeq Y(\fg)^{\fg_0} \otimes_{U(\fg_0)^{\fg_0}} U(\fg_0);$
\item $\fz_{Y(\fg)}(ZU(\fg_0)) = Y(\fg)^{\fg_0} \otimes_{U(\fg_0)^{\fg_0}} U(\fg_0).$
\end{enumerate}
%Particularly, 
%$$\fz_{Y(\fg)}(ZU(\fg)) = Y(\fg)^{\fg} \otimes_{Z(U(\fg))} U(\fg).$$
\end{lem}
\begin{proof}
1) Let us consider the associated bigraded algebra with respect to the filtrations $F_1, F_2$. Then $\gr_{21} Y(\fg)^{\fg_0} = S(\fg[t])^{\fg_0}, \gr_{21} U(\fg_0)^{\fg_0} = S(\fg_0)^{\fg_0}, \gr_{21} U(\fg_0) = S(\fg_0)$. 

Let $f_n(t)$ be a polynomial of degree $n$ with $n$ pairwise different roots. Consider the quotient $S_n(\fg[t]) : = S(\fg[t]) / f_n(t)$.
From \cite[Lemma 4.8]{ir} it follows that $$S_n(\fg[t])^{\fg_0} \cdot S(\fg_0)  \simeq S_n(\fg[t])^{\fg_0} \otimes_{S(\fg_0)^{\fg_0}} S(\fg_0).$$

We see that the statement of Lemma holds for any filtered component and hence for the whole algebra.

2) It follows from \cite[Main Theorem (d)]{knop} if we consider the associated graded with respect to filtration $F_2$ and follow the proof of first statement.
\end{proof}

In our limit we have $B(E) \subset Y(\fg)^{\fg}$ and it is a maximal subalgebra of $Y(\fg)^{\fg}$ (Lemma \ref{max1}). Moreover, $Z(U(\fg))$ lie in $B(E)$. From Lemma~\ref{le:knop} we see that the limit subalgebra lies in $B(E) \cdot U(\fg) \simeq B(E)\otimes_{ZU(\fg)}U(\fg)$. But then it should lie in the centralizer of the quadratic part of the Bethe subalgebra in the latter tensor product i.e. in $B(E) \cdot \mathcal{A}_{\chi} \simeq B(E)\otimes_{ZU(\fg)}\mathcal{A}_{\chi}$.

Subalgebra $B(E)\otimes_{ZU(\fg)}\mathcal{A}_{\chi}$ has the same Poincar\'e series as $B(C), C \in T^{reg}$. Then the limit coincides with   $B(E)\otimes_{ZU(\fg)}\mathcal{A}_{\chi}$.
\end{proof}

\begin{thm}
Let $C(\varepsilon) = C_0 \exp(\varepsilon \chi), C_0 \in T \setminus T^{reg}$ with $\chi \in \fh\subset \fz_\fg(C_0)$ being a generic regular semisimple element of the centralizer of $C_0$.
Then $$\lim_{\varepsilon \to 0} B(C(\varepsilon)) = B(C_0) \otimes_{Z(U(\fz_{\fg}(C_0)))} \mathcal{A_{\chi}},$$
where $\mathcal{A_{\chi}} \subset U(\fz_{\fg}(C_0))$ is the quantum shift of argument subalgebra corresponding to $\chi$.

\end{thm}

\begin{proof}
The proof is the same as the proof of previous Theorem, with the only difference that we use Corollary \ref{maxc} instead of Corollary \ref{max1}.

\end{proof}

\begin{rem}
\emph{One can solve Vinberg's problem of lifting shift of argument subalgebras for generic $\chi\in\fh$ to the universal enveloping algebra by \emph{defining} the lifting of $A_\chi\subset S(\fg)$ to the universal enveloping algebra as $\mathcal{A}_\chi:=U(\fg)\cap \lim_{\varepsilon \to 0} B(C(\varepsilon))$ for $C(\varepsilon)=\exp(\varepsilon \chi)$.}
\end{rem}

\bigskip
\footnotesize{
{\bf Leonid Rybnikov} \\
National Research University
Higher School of Economics,\\ Russian Federation,\\
Department of Mathematics, 6 Usacheva st, Moscow 119048;\\
Institute for Information Transmission Problems of RAS;\\
{\tt leo.rybnikov@gmail.com}} \\
\\
\footnotesize{
{\bf Aleksei Ilin} \\ National Research University
Higher School of Economics, \\
Russian Federation,\\
Department of Mathematics, 6 Usacheva st, Moscow 119048;\\
{\tt alex\_omsk@211.ru}}

\end{document}